\documentclass[a4paper,11pt,centertags,reqno]{amsart}

\usepackage{letltxmacro}        
\usepackage{ifthen}							
\usepackage[headings]{fullpage} 
\usepackage{setspace}						
\usepackage[inline]{enumitem}   
\usepackage{array}							
\usepackage{fancybox}						
\usepackage{graphicx}           
\usepackage{subfigure}					
\usepackage{amsmath}	        	
\usepackage{amsthm}							
\usepackage{amssymb}			
\usepackage{mathtools}					
\usepackage{mathrsfs}						
\usepackage{stmaryrd}						
\usepackage{yhmath}							
\usepackage{accents}						
\usepackage{extarrows}          
\usepackage{xfrac}              
\usepackage{url}
\usepackage[all]{xy}						
\usepackage{comment}
\usepackage[normalem]{ulem}
\usepackage[makeroom]{cancel}
\usepackage[square,numbers]{natbib} 
\usepackage[multiple]{footmisc}
\usepackage{lmodern}
\usepackage{anyfontsize}
\PassOptionsToPackage{hyphens}{url}
\usepackage[colorlinks,citecolor=blue,urlcolor=blue,unicode,verbose]{hyperref}
\usepackage{color}
\DeclareFontFamily{U}{mathx}{\hyphenchar\font45}
\DeclareFontShape{U}{mathx}{m}{n}{
      <5> <6> <7> <8> <9> <10>
      <10.95> <12> <14.4> <17.28> <20.74> <24.88>
      mathx10
      }{}
\DeclareSymbolFont{mathx}{U}{mathx}{m}{n}
\DeclareFontSubstitution{U}{mathx}{m}{n}
\DeclareMathAccent{\widecheck}{0}{mathx}{"71}
\DeclareMathAccent{\wideparen}{0}{mathx}{"75}

\makeatletter
\newcommand*\rel@kern[1]{\kern#1\dimexpr\macc@kerna}
\newcommand*\widebar[1]{%
  \begingroup
  \def\mathaccent##1##2{%
    \rel@kern{0.8}%
    \overline{\rel@kern{-0.8}\macc@nucleus\rel@kern{0.2}}%
    \rel@kern{-0.2}%
  }%
  \macc@depth\@ne
  \let\math@bgroup\@empty \let\math@egroup\macc@set@skewchar
  \mathsurround\z@ \frozen@everymath{\mathgroup\macc@group\relax}%
  \macc@set@skewchar\relax
  \let\mathaccentV\macc@nested@a
  \macc@nested@a\relax111{#1}%
  \endgroup
}
\makeatother

\newcommand{\N}{\mathbb{N}}

\newcommand{\R}{\mathbb{R}}

\newcommand{\Pcal}{\mathcal{P}}
\newcommand{\Bcal}{\mathcal{B}}

\newcommand{\Hcal}{\mathcal{H}}

\newcommand{\Ocal}{\mathcal{O}}

\newcommand{\loc}{\rm{loc}}

\renewcommand{\C}[2]{
  \ifthenelse{#1=0 \and #2=0}{\textsf{\upshape C}}
  {\ifthenelse{#2=0}{\textsf{\upshape C}^{#1}}
    {\textsf{\upshape C}^{#1,#2}}}
}

\renewcommand{\d}{\mathrm{d}}
\newcommand{\ddp}{\mathrm{dp}}
\newcommand{\q}{\mathrm{qc}}

\newcommand{\E}{\textsf{\upshape E}}

\renewcommand{\P}{\textsf{\upshape P}}
\newcommand{\Qu}{\textsf{\upshape Q}}

\newcommand{\indicator}[1]{\mathbf{1}_{#1}}

\newcommand{\filt}[1]{\mathfrak{#1}}
\newcommand{\sigalg}[1]{\mathscr{#1}}

\renewcommand{\S}{\ensuremath{\mathscr{S}}}

\newcommand{\V}{\ensuremath{\mathscr{V}}}
\newcommand{\J}{\ensuremath{\mathscr{J}}}

\newcommand{\lc}{[\![}
\newcommand{\rc}{]\!]}

\DeclarePairedDelimiterX\br[1]{[}{]}{#1}

\DeclarePairedDelimiterX\set[2]\{\}{#1\::\:#2}

\makeatletter
\let\oldr@@t\r@@t
\def\r@@t#1#2{%
  \setbox0=\hbox{$\oldr@@t#1{#2\,}$}\dimen0=\ht0
  \advance\dimen0-0.2\ht0
  \setbox2=\hbox{\vrule height\ht0 depth -\dimen0}%
  {\box0\lower0.4pt\box2}}
\LetLtxMacro{\oldsqrt}{\sqrt}
\renewcommand*{\sqrt}[2][\ ]{\oldsqrt[#1]{#2}}
\makeatother

\theoremstyle{plain}
\newtheorem{theorem}{Theorem}
\newtheorem{lemma}[theorem]{Lemma}
\newtheorem{proposition}[theorem]{Proposition}
\newtheorem{corollary}[theorem]{Corollary}
\theoremstyle{definition}
\newtheorem{definition}[theorem]{Definition}

\newtheorem{example}[theorem]{Example}

\theoremstyle{remark}
\newtheorem{remark}[theorem]{Remark}

\numberwithin{theorem}{section}
\numberwithin{equation}{section}
\numberwithin{figure}{section}
\numberwithin{table}{section}

\graphicspath{{Figures/}}
\allowdisplaybreaks[4]

\begin{document}
\title{Pure-Jump Semimartingales}

\author{Ale\v{s} \v{C}ern\'{y} \and Johannes Ruf}

\address{Ale\v{s} \v{C}ern\'{y}\\
  Business School (formerly Cass)\\
	City, University of London}

\email{ales.cerny.1@city.ac.uk}

\address{Johannes Ruf\\
  Department of Mathematics\\
  London School of Economics and Political Science}

\email{j.ruf@lse.ac.uk}



\date{\today}

\begin{abstract} 
A new integral with respect to an integer-valued random measure is introduced. In contrast to the finite variation integral  ubiquitous in semimartingale theory (Jacod and Shiryaev \citep[II.1.5]{js.03}), the new integral is closed under stochastic integration, composition, and smooth trans\-for\-ma\-tions. The new integral gives rise to a previously unstudied class of pure-jump processes --- the sigma-locally finite variation pure-jump processes. As an application, it is shown that every semimartingale $X$ has a unique decomposition 
$$X = X_0 + X^{\q}+X^{\ddp},$$ 
where $X^{\q}$ is quasi-left-continuous and $X^\ddp$ is a sigma-locally finite variation pure-jump process that jumps only at predictable times, both starting at zero. The decomposition mirrors the classical result for local martingales (Yoeurp \cite[Theoreme~1.4]{yoeurp.76}) and gives a rigorous meaning to the notions of continuous-time and discrete-time components of a semimartingale.   Against this backdrop, the paper investigates a wider class of processes that are equal to the sum of their jumps in the semimartingale topology and constructs a taxonomic hierarchy of pure-jump semimartingales. 
\end{abstract}

\maketitle

\section{Introduction}
Denote by $\V^{\d}$ the set of finite variation pure-jump semimartingales, i.e., those $X$ that are equal to their initial value plus the absolutely convergent sum of their jumps. Equivalently, $X\in\V^{\d}$ if 
$$X=X_0+x*\mu^X,$$ 
where $\mu^X$ is the jump measure of $X$ and $x*\mu^X$ represents the standard jump measure integral (Jacod and Shiryaev \cite[II.1.5]{js.03}). It is known (see \cite[II.1.28]{js.03}) that the predictable compensator of $X\in\V^{\d}$ equals $x*\nu^X$ whenever such compensator exists. Here $\nu^X$ denotes the predictable compensator of the jump measure $\mu^X$.

Consider now the $\R$--valued stochastic process $X$ defined by the following properties: 
\begin{equation}\label{eq:Ex}
	\begin{rcases}
		&\text{$X_0=0$;}\\
		&\text{$X$ has independent increments;}\\
		&\text{jumps of $X$ occur only at fixed times $\tau_n = 2-\sfrac{1}{n}$, for each $n\in\N$;}\quad\\
		&\text{the process jumps by $\pm\sfrac{1}{n}$ with equal probability, for each $n\in\N$;} \\
		&\text{$X$ remains constant outside the fixed jump times.} 
	\end{rcases}
\end{equation}
This process is a well-defined semimartingale, in fact a square-integrable martingale, on the whole time line $[0,\infty]$.  
Yet $X$ is not equal to the sum of its jumps in the conventional sense because the jumps of $X$ are not absolutely summable. In particular, the standard integral $x*\mu^X$ diverges. Furthermore, the integral $x*\nu^X$ also diverges even though the predictable compensator (drift) of $X$ exists and is equal to zero.

The seminal work of Kallsen~\cite{kallsen.04} offers two important clues how to perform  symbolic drift calculation for the process in \eqref{eq:Ex}. The first key idea, \cite[Definition~4.1]{kallsen.04}, is to consider each jump time separately and ask only that the expectation $\E[|\Delta X_{\tau_n}|]$ is finite (for each $n\in\N$). One may then be able to legitimately sum up the individual contributions $\E[\Delta X_{\tau_n}]$ over $n\in\N$ even if the sum of $\E[|\Delta X_{\tau_n}|]$ does not converge, which is precisely the case for \eqref{eq:Ex}. The second key idea, \cite[Lemma~4.1]{kallsen.04}, is non-trivial mathematically: it identifies the natural procedure above with the sigma-localization of the absolutely convergent integral $x*\nu^X$ to obtain a better integral that we shall denote $x\star\nu^X$. It is important to add that Kallsen \cite{kallsen.04} makes everything work in a general continuous-time setting that goes well beyond \eqref{eq:Ex}.

Acting on these clues, we propose to view the process \eqref{eq:Ex} in two novel, complementary ways, both of which involve an approximation by elements in $\V^{\d}$. The first approach regards $X$ in \eqref{eq:Ex} as an element of $\V^{\d}_\sigma$, i.e., as a process that belongs sigma-locally to $\V^{\d}$. Proposition~\ref{P:190729.2} shows that each $X\in\V^{\d}_\sigma$ satisfies the convenient formula 
$$X = X_0 + x\star\mu^X,$$ 
where $\star$ is the sigma-localized version of the standard jump measure integral~$*$. Furthermore, the drift of $X$, provided it exists, is given by $x\star \nu^X$ (Corollary~\ref{C:drift}). 

The new jump integral $\star$, unlike $*$, is associative so that for a predictable process $\zeta$ and a predictable function $\eta$ one has 
$$\zeta \cdot (\eta \star \mu^X) = (\zeta \eta) \star \mu^X$$ 
if the left-hand side is well defined (Proposition~\ref{P:190411}). The associativity of the stochastic integral is in fact a special case of a stronger composition property 
$$ \psi\star(\eta\star\mu^X) = \psi(\eta)\star\mu^X,$$
which holds as soon as one of the two expressions is meaningful (Proposition~\ref{P:190528.1}). The associativity of integral then follows by taking for $\psi$ the linear function $\psi(x) = \zeta x$. Together with closedness under $\mathcal{C}^2$ transformations (Proposition~\ref{P:190528.2}) these properties give rise to a very pleasant stochastic calculus for processes in $\V^{\d}_\sigma$. For further developments in this direction, see \v{C}ern\'y and Ruf \cite{crII}.

The second approach views $X$ as a sum of its jumps at a sequence of stopping times with convergence in the \'Emery semimartingale topology. Here it is in principle possible to encounter two different processes that share the same jump measure (by choosing different exhausting sequence of stopping times in each case). However, we show that such a situation cannot occur in $\V^{\d}_\sigma$ (Theorem~\ref{T:1}). This offers a wider sense in which all processes in $\V^{\d}_\sigma$ are uniquely determined by their jump measure. In contrast, the family of quadratic pure-jump processes (Protter \cite[p.~63]{protter.90}; see also Definition~\ref{D.1}) lacks such uniqueness property as adding a  continuous deterministic finite variation process to $X$ yields a different quadratic pure-jump process with the same jump measure. Thus the terminology ``quadratic pure-jump'' is unfortunate --- in the context of this paper, most of these processes are not pure-jump processes at all. In fact, it is shown that a special quadratic pure-jump semimartingale is in $\V^{\d}_\sigma$ if and only if its drift is given by $x\star\nu^X$ (Corollary~\ref{C:1}).

The process~\eqref{eq:Ex} has one property that not all processes in $\V^{\d}_\sigma$ share: its jumps are exhausted by a sequence of \emph{predictable} (in fact fixed) times. Such processes enjoy further regularity in that \emph{any} sequence of predictable times exhausting their jumps delivers convergence of the sum in the \'Emery topology (Proposition~\ref{P:190828.2}). As an application, it is shown that every semimartingale $X$ has a unique decomposition 
$$X = X_0 + X^{\q}+X^{\ddp},$$ 
where $X^{\q}$ is quasi-left-continuous and $X^\ddp$ is an element of $\V^{\d}_\sigma$ that jumps only at predictable times, both starting at zero (Proposition~\ref{P:190729}). The decomposition mirrors the classical result for local martingales (Yoeurp \cite[Theoreme~1.4]{yoeurp.76}) and gives a rigorous meaning to the notions of continuous-time and discrete-time components of a semimartingale. The decomposition is helpful when computing drifts in various applications; for further details we refer the reader to \v{C}ern\'y and Ruf \cite[Section~6]{crII}.

Here now the outline of this paper. After notation (Section~\ref{S:2}), we develop the new integral (Section~\ref{S:3}) and then proceed with the classification of the pure-jump processes (Section~\ref{S:4}). Section~\ref{S:5} has proofs pertaining to Section~\ref{S:4}. Appendix~\ref{S:Emery} contains facts about the \'Emery topology and Appendix~\ref{S:6} discusses the consequences of weakening the topology to the one given by uniform convergence in probability.

\section{Notation and setup} \label{S:2}

We fix a probability space $(\Omega ,\sigalg{F},\P)$ with a right-continuous filtration $\filt{F} = (\sigalg{F}_t)_{t \geq 0}$. 
 Recall from \cite[II.1.4]{js.03} that a function $\eta: \Omega \times [0,\infty) \times \R \rightarrow \R$ is called predictable if it is $\Pcal \times \Bcal(\R)$--measurable, where $\Pcal$ denotes the predictable sigma field and $\Bcal(\R)$ the Borel sigma field on $\R$. 
If  $\psi: \Omega \times [0,\infty) \times \R \rightarrow \R$ 
denotes another (predictable) function we shall write $\psi(\eta)$ to denote the (predictable) function $(\omega, t, x) \mapsto \psi(\omega, t, \eta(\omega, t, x))$.

We shall consider  an integer-valued random measure $\mu$ on $[0,\infty) \times \R$ with predictable compensator $\nu$.  A predictable function $\eta$ with $\eta(0) = 0$ is integrable with respect to $\mu$, i.e., $\eta * \mu$ exists if $|\eta| * \mu < \infty$. Recall from \cite[II.2.9]{js.03} that $\nu$ can be written in disintegrated form as
\begin{align} \label{eq:190528}
	\nu(\d t, \d x) = F_t(\d x) \d A_t, \qquad t \geq 0,\, x \in \R,
\end{align}
where $A$ is a nondecreasing predictable 
process, and $F$ is a transition kernel from $(\Omega \times [0,\infty), \Pcal)$ into $(\R, \Bcal(\R))$.   If we want to emphasize the probability measure under which $\nu$ is the predictable compensator of $\mu$ we shall write $\nu(\P)$.  

We let $\S$ denote the space of $\R$--valued semimartingales. 
For a semimartingale $X \in \S$, we let $X_-$ denote its left-limit process with the convention $X_{0-} = X_0$ and we let $\Delta X = X - X_-$ denote its jump process. 
 Next, we let $\mu^X$ denote the jump measure of $X$ and $\nu^X$ its predictable compensator.  For a predictable function $\eta$ with $\eta(0) = 0$ we then have $\eta * \mu^X = \sum_{0<t \leq \cdot} \eta(\Delta X_t)$ if $|\eta| * \mu^X < \infty$. The corresponding quantities in \eqref{eq:190528} shall be written with a superscript $X$.
 If $X$ is special, we write $B^X$ for its drift, i.e., $B^X$ is the predictable finite variation process with $B^X_0 = 0$ such that $X-B^X$ is a local martingale. If $Y$ denotes another semimartingale then $[X,Y]$ denotes the quadratic covariation of $X$ and $Y$. Moreover, we write $X[1] = X - x\indicator{|x| > 1} * \mu^X$ and note that $X[1]$ is special. Next, $L(X)$ denotes the family of $X$--integrable predictable processes.
 
If $\J \subset \S$ is a family of semimartingales, we say that a semimartingale $Y$ belongs to $\J_\sigma$, the 
sigma--localized class of $\J$, 
if there is a sequence $(D_n)_{n \in \N}$ of predictable sets increasing to 
$\Omega\times[0,\infty)$  such that $\indicator{D_n} \cdot Y \in \J$ for each $n\in\N$. We say that $\J$ is stable under sigma-stopping 
(see Kallsen \cite[Definition~2.1]{kallsen.04}) if for every $X\in\J$ and every predictable set $D$ the process $\indicator{D}\cdot X$ belongs to $\J$. Finally, we shall say that $\Qu$ is a probability measure that is locally absolutely continuous with respect to $\P$ if $\Qu$ is absolutely continuous with respect to $\P$ on $\sigalg{F}_t$ for each $t \geq 0$.

\begin{remark}
	Throughout this paper, we only consider the scalar case, which helps in reducing notation. The careful reader can convince themselves that quite a few results (in particular those of Section~\ref{S:3}) generalize to the higher-dimensional case, for example when $X$ takes values in $\R^d$ and the predictable functions below  map into $\R^n$, etc., for some $d,n \in \N$. A notable exception is statement~\ref{T:1.ii} in Theorem~\ref{T:1}, where we do not know whether the one-dimensional situation generalizes. Indeed, we do not know whether $\J^2$ is a vector space --- the lack of such structure would seem to imply that such a result does not hold in higher dimensions.
\qed
\end{remark}

\section{Extended integral with respect to a random measure}  \label{S:3}

We start by extending the standard definition of integral with respect to a random measure and derive some basic properties in Subsection~\ref{SS:3.1}. Then, in Subsection~\ref{SS:3.2}, we prove some  associativity properties of this integral. In Subsection~\ref{SS:3.3} we connect the integral to the representation of sigma-locally finite variation pure-jump processes. In particular, this will enable us to write the process $X$ of the introduction as $X=x\star\mu^X$ and its drift under any locally absolutely continuous measure $\Qu$ that makes $X$ special as $B^X(\Qu) = x\star\nu^X(\Qu)$.

\subsection{Definition and basic properties of the extended integral}  \label{SS:3.1}

\begin{definition}[Extended integral with respect to random measure] \label{D:190405}\mbox{}
\begin{enumerate}[label={\rm(\roman{*})}, ref={\rm(\roman{*})}] 
\item \label{D:190405:i}   Denote by $L(\mu)$ the set of predictable functions that are absolutely integrable with respect to $\mu$. We say that a predictable function $\eta$ belongs to $L_\sigma(\mu)$, the sigma--localized class of $L(\mu)$, 
if there is a sequence $(D_n)_{n \in \N}$ of predictable sets increasing to $\Omega\times[0,\infty)$ and a semimartingale $Y$ such that $\indicator{D_n}\eta\in L(\mu)$ for each $n\in\N$ and 
$$ (\indicator{D_n}\eta)*\mu = \indicator{D_n}\cdot Y,\qquad n\in\N.$$
In such case the semimartingale $Y$ is denoted by $\eta \star \mu$.
\item \label{D:190405:ii} $L(\nu)$, $L_\sigma(\nu)$, and $\eta\star\nu$ are defined analogously; 
see Definition~4.1 and Lemma~4.1 in Kallsen \cite{kallsen.04}.
\qed
\end{enumerate}
\end{definition}

Note that if $\mu=\nu$ is a predictable random measure then the two definitions above agree; hence $L_\sigma(\mu)$ and $L_\sigma(\nu)$ are well-defined and we have $L_\sigma(\mu) = L_\sigma(\nu)$. Note also that $\eta \star \mu$ (resp., $\eta \star \nu$) is uniquely defined provided that $\eta \in L_\sigma(\mu)$ (resp., $\eta \in L_\sigma(\nu)$). 

\begin{remark} \label{R:190405}
	Let $\Qu$ denote a probability measure locally absolutely continuous  with respect to $\P$. With the obvious notation, we then have  $L_\sigma^\P(\mu) \subset L_\sigma^\Qu(\mu)$. For $L_\sigma^\P(\nu(\P))$ and $L_\sigma^\Qu(\nu(\Qu))$, no such inclusions hold in general. However, refer also to the positive statement in Remark~\ref{R:190507}.	
\qed
\end{remark}

The following characterization of $L_\sigma(\nu)$ appears in the literature. 

\begin{lemma}[Kallsen {\cite[Definition~4.1 and Lemma~4.1]{kallsen.04}}]  \label{L:Kallsen}
For a predictable function $\eta$ the following statements are equivalent.
\begin{enumerate}[label={\rm(\roman{*})}, ref={\rm(\roman{*})}] 
\item $\eta\in L_\sigma(\nu)$.
\item The following two conditions hold.
	\begin{enumerate}
		\item $\int|\eta_t(x)| F_t(\d x) < \infty \quad (\P \times \d A)$--a.e. 
		\item $\int \left| \int \eta_t(x) F_t(\d x) \right|  \d A_t < \infty$.
	\end{enumerate}
\end{enumerate}
Moreover, for $\eta\in L_{\sigma}(\nu)$ one has 
\[ 
\eta \star \nu =\int_0^\cdot \left(\int \eta_t( x) F_t(\d x) \right) \d A_t.
\]
\end{lemma}

To the best of our knowledge, the class $L_\sigma(\mu)$ has not been studied previously. 
The following characterization therefore seems to be new. 
\begin{proposition}\label{P:190330}
 For a predictable function $\eta$ the following statements are equivalent.

\begin{enumerate}[label={\rm(\roman{*})}, ref={\rm(\roman{*})}]
\item\label{P:190330.i} $\eta\in L_\sigma(\mu)$.
\item\label{P:190330.ii} The following two conditions hold.
		\begin{enumerate}
				\item $\eta^2 * \mu < \infty$.
				\item $\eta \indicator{\{|\eta| \leq 1\}} \in L_\sigma(\nu)$. 
			\end{enumerate}
\end{enumerate}
Furthermore, for $\eta \in L_{\sigma}(\mu)$ one has 
\begin{equation} \label{eq:190405.1}
\eta \star \mu =\eta\indicator{\{|\eta|> 1\}}*\mu+\eta\indicator{\{|\eta|\leq 1\}}*(\mu-\nu)+\eta\indicator{\{|\eta|\leq 1\}}\star\nu,
\end{equation}
where the integral with respect to the compensated measure $\mu-\nu$ is defined in \cite[II.1.27(b)]{js.03}.
\end{proposition}

\begin{remark}\label{R:190507}
	In the setup of Remark~\ref{R:190405}, choose a predictable function $\eta$ with $|\eta|^2 * \mu < \infty$. Proposition~\ref{P:190330} now yields that if $\eta \indicator{\{|\eta| \leq 1\}} \in L_\sigma^\P(\nu(\P))$ then also $\eta \indicator{\{|\eta| \leq 1\}} \in L_\sigma^\Qu(\nu(\Qu))$.
\qed
\end{remark}

\begin{proof}[Proof of Proposition~\ref{P:190330}] In the following we argue both inclusions and \eqref{eq:190405.1}.

\ref{P:190330.i}$\Rightarrow$\ref{P:190330.ii}:  Let $(D_n)_{n \in \N}$ be as in Definition~\ref{D:190405}\ref{D:190405:i}. Then 
$\indicator{D_n} |\eta|^2 *\mu=  \indicator{D_n} \cdot [\eta \star \mu,\eta \star \mu]$ for all $ n \in \N$,
and a monotone convergence argument yields $|\eta|^2 *\mu = [\eta \star \mu,\eta \star \mu] < \infty$. Let us now set 
$\overline \eta = \eta  \indicator{\{|\eta| \leq 1\}}$. Then $\overline \eta \in L_\sigma(\mu)$ and we directly get
\[
	\int_0^\cdot  \indicator{D_n}(t) \left(\int|\overline \eta_t(x)| F_t(\d x)\right) \d A_t = \indicator{D_n} |\overline \eta|*\nu < \infty.
\]
Thanks to Lemma~\ref{L:Kallsen} we now only need to argue that $\int \left| \int \overline \eta_t(x) F_t(\d x) \right|  \d A_t < \infty$. 
We 	note that $|\Delta (\overline \eta \star \mu)| \leq 1$, hence $\overline \eta \star \mu$ is special, say with predictable finite variation drift $\overline B$. By monotone convergence, we now get 
\begin{align*}
		\int_0^\cdot \left| \int \overline \eta_t(x) F_t(\d x)\right| \d A_t  
		&= \lim_{n \uparrow \infty} \int_0^\cdot \indicator{D_n}(t) \left| \int \overline \eta_t(x) F_t(\d x)\right| \d A_t 
		=  \lim_{n \uparrow \infty} \int_0^\cdot \indicator{D_n}(t)  |\d \overline B_t|  \\
		&=  \int_0^\cdot  |\d \overline B_t| < \infty.
\end{align*}
This yields $\overline \eta \in L_\sigma(\nu)$, hence the implication \ref{P:190330.i}$\Rightarrow$\ref{P:190330.ii} is shown.
	
\ref{P:190330.ii}$\Rightarrow$\ref{P:190330.i} and \eqref{eq:190405.1}: Let $(D_n)_{n \in \N}$ be as in 
Definition~\ref{D:190405}\ref{D:190405:ii}.	Note that all terms on the right-hand side of \eqref{eq:190405.1} are well defined and yield a semimartingale $Y$ provided that \ref{P:190330.ii} holds. Thanks to the uniqueness of $\eta \star \mu$ we only need to observe that 
$\indicator{D_n} \cdot Y =  (\indicator{D_n} \eta) * \mu$ for all $n \in \N$. However, this is straightforward, which concludes the proof of the proposition.
\end{proof}
	
\begin{remark}
	Note that $L_{\loc}(\mu) = L(\mu)$, that is $L(\mu)$ is closed under standard localization. However, we have $L_{\sigma}(\mu) \supsetneq L(\mu)$ on sufficiently large probability spaces; see Example~\ref{ex:190729}.
\qed
\end{remark}

\begin{example} \label{ex:190729}
	Let $\mu$ denote a jump measure with  $\nu(\d t, \d x) = \sfrac{1}{2}  \indicator{|x| \leq 1} \d x (\sum_{i = 1}^\infty \delta_{\sfrac{1}{i}} (\d t))$, where $\delta_{\sfrac{1}{i}}$ denotes the Dirac measure at $\sfrac{1}{i}$. Then for the decomposition in \eqref{eq:190528} we may choose $A = \sum_{i = 1}^\infty \sfrac{1}{i^2} \indicator{\lc \sfrac{1}{i}, \infty\lc}$ and $F_t(\d x) = \sfrac{1}{(2 t^2)} \indicator{|x| \leq 1} \d x$ for all $t > 0$ and $x \in \R$.

Consider next the predictable function $\eta$ given by $\eta_t(x) = t x$ for all $t \geq 0$ and $x \in \R$. Then $\eta \in L_\sigma(\nu)$ by Lemma~\ref{L:Kallsen} and $\eta^2 * \mu_\infty \leq \sum_{i = 1}^\infty \sfrac{1}{i^2} < \infty$.  Hence by Proposition~\ref{P:190330}, we have $\eta \in L_\sigma(\mu)$. Indeed, $\eta \star \mu$ is a semimartingale that jumps at times $\sfrac{1}{i}$ and is constant on the intervals $[\sfrac{1}{(i+1)}, \sfrac{1}{i})$,  for each $i \in \N$. Moreover, 	 
	 $\Delta(\eta \star \mu)_{\sfrac{1}{i}} = \sfrac{U_i}{i}$, where $(U_i)_{i \in \N}$ is a sequence of independent $[-1,1]$-uniforms.
 Since $|\eta| * \mu_1 = \sum_{i = 1}^\infty \sfrac{|U_i|}{i} = \infty$ (by Kolmogorov's convergence criteria), we have $\eta \in L_\sigma(\mu) \setminus L(\mu)$.
	
Consider now the predictable function $\bar \eta_t(x) = x$  for all $t \geq 0$ and $x \in \R$.  Then again $\bar \eta \in L_\sigma(\nu)$, but now $\bar\eta^2 * \mu_1 = \infty$.  This yields an example for a predictable function 
		$\bar\eta \in L_\sigma(\nu) \setminus L_\sigma(\mu)$.
	\qed
\end{example}

\subsection{Associativity properties of the extended integral}	\label{SS:3.2}

We remind the reader that $\mu$ without a superscript refers to a given integer-valued random measure, while $\mu^X$ refers to the jump measure of a semimartingale $X$; see Section~\ref{S:2}. 
\begin{proposition}\label{P:190528.1}
	Let $\eta \in L_\sigma(\mu)$ and $\psi: \Omega \times [0,\infty) \times \R \rightarrow \R$ be a predictable function. Then the following statements are equivalent.
\begin{enumerate}[label={\rm(\roman{*})}, ref={\rm(\roman{*})}]
\item\label{P:190528.i} $\psi\in L_\sigma(\mu^{\eta\star\mu})$.
\item\label{P:190528.ii} $\psi ( \eta) \in L_\sigma(\mu)$.
\end{enumerate}
Furthermore, if either condition holds then $\psi \star (\eta \star \mu) = \psi(\eta) \star \mu$. 
Moreover, the same assertions hold with $\mu$ replaced by $\nu$.
\end{proposition}

\begin{proof}
	Let us first prove the statement with $\mu$ replaced by $\nu$.  To this end, note that 
\begin{align*}
	\nu^{\eta \star \nu} (\d t, \d x) =  \overline F_t (\d x) \d A_t, \qquad t \geq 0,\, x \in \R,  
\end{align*}
where $\overline F_t$ is the image of measure $F_t$ under $\eta_t$. Then the equivalence follows from Lemma~\ref{L:Kallsen}. The statement for $\mu$ follows exactly in the same manner, now using Proposition~\ref{P:190330}.
\end{proof}

Next, we prove a composition property for stochastic integrals. Recall that $L(\eta\star\mu)$ denotes the set of predictable processes that are integrable with respect to the semimartingale $\eta\star\mu$.
\begin{proposition}\label{P:190411}
Let $\eta \in L_\sigma(\mu)$ and $\zeta: \Omega \times [0,\infty) \rightarrow \R$ be a predictable process. Then the following statements are equivalent. 

\begin{enumerate}[label={\rm(\roman{*})}, ref={\rm(\roman{*})}]
\item\label{P:190411.i} $\zeta\in L(\eta\star\mu)$.
\item\label{P:190411.ii} $\zeta\eta\in L_\sigma(\mu)$.
\end{enumerate}
Furthermore, if either condition holds then
$\zeta\cdot(\eta\star\mu) = (\zeta\eta)\star\mu$.
Moreover, the same assertions hold with $\mu$ replaced by $\nu$.
\end{proposition}

\begin{proof}
We shall prove the statement only for $\mu$ as the same argument works if $\mu$ is replaced by $\nu$.
Note that there is a sigma--localizing sequence $(D_n)_{n \in \N}$ such that 
\[
D_n \subset \left\{(\omega,t)\in \Omega \times [0, \infty)  :  |\zeta(\omega,t)| \leq n \right\} 
\]
and $\indicator{D_n}\eta\in L(\mu)$ with $\indicator{D_n}\cdot(\eta\star\mu) = (\indicator{D_n}\eta)*\mu.$
 	Note that for each $n \in \N$ we have $\zeta \indicator{D_n} \in L(\eta\star\mu)$ and $\zeta \indicator{D_n} \eta\in L(\mu)$.
	Then if \ref{P:190411.i}  holds we have 
	\begin{align*}
\indicator{D_n}\cdot\left(\zeta\cdot(\eta\star\mu)\right) &= (\zeta \indicator{D_n}) \cdot (\indicator{D_n} \cdot(\eta\star\mu)) 
= (\zeta \indicator{D_n})  \cdot (\indicator{D_n}\eta * \mu)
= (\zeta \indicator{D_n} \eta) * \mu,
\end{align*}	
for each $n \in \N$. This yields  \ref{P:190411.ii} and $\zeta\cdot(\eta\star\mu) = (\zeta\eta)\star\mu$.

Assume now that \ref{P:190411.ii}  holds. We then have
	\begin{align*}
\indicator{D_n}\cdot\left((\zeta \eta)\star\mu\right) &{}= \left(\indicator{D_n}\zeta \eta\right)*\mu 
                                                         = (\zeta \indicator{D_n}) \cdot (\indicator{D_n}\eta * \mu)\\
                                                      &{}= (\zeta \indicator{D_n}) \cdot (\indicator{D_n} \cdot (\eta \star \mu))
																											   = (\zeta \indicator{D_n}) \cdot (\eta \star \mu)
\end{align*}	
for each $n \in \N$. An application of \cite[Lemma~2.2]{kallsen.04} then yields that \ref{P:190411.i} holds.
\end{proof}

\begin{proposition}\label{P:190528.2}
Let $\eta \in L_\sigma(\mu)$ and $f: \R \rightarrow \R$ be a twice continuously differentiable function. Then for $Y = Y_0 + \eta \star \mu$ we have $\xi = f(Y_- + \eta) - f(Y_-) \in L_\sigma(\mu)$ and $f(Y) = f(Y_0) + \xi \star \mu$.
\end{proposition}
\begin{proof} 
For each $n \in \N$ let $\tau_n$ denote the first time that $|Y|$ is greater than or equal to $n$. 
	Recall that $f(Y)$ is a semimartingale and that there exists a sequence $(D_n)_{n \in \N}$  of predictable sets increasing to $\Omega \times [0, \infty)$ such that $\indicator{D_n} \eta \in L(\mu)$. Without loss of generality we may assume that $D_n \subset \lc 0, \tau_n\rc$ for each $n \in \N$. 
	
It suffices now to argue that  $\indicator{D_n} \xi \in L(\mu)$ and that $\indicator{D_n} \cdot f(Y) = (\indicator{D_n} \xi) * \mu$ for each $n \in \N$. To argue these two claims, fix $n \in \N$. Observe now that
	\[
		\indicator{D_n} |\xi| \leq \indicator{D_n} \indicator{\{|\eta| > 1\}} |\xi| 
			+  L_{n+1}  \indicator{D_n} \indicator{\{|\eta| \leq 1\}} | \eta|,
	\] 
	where $L_{n+1}$ is the Lipschitz constant of $f$ restricted to the domain $[-n-1, n+1]$. 
	The first term on the right-hand side of the last display is in $L(\mu)$ as it only contributes finitely many terms and so is the second term as it is bounded by a constant times $\indicator{D_n} |\eta|$. Hence we have argued the first claim, namely that $\indicator{D_n} \xi \in L(\mu)$.   Next, note that It\^o's formula and Proposition~\ref{P:190528.1} yield  
	\begin{align*}
		\indicator{D_n} \cdot f(Y) &= \indicator{D_n} \cdot 
			\left( f'(Y_-) \cdot Y + (\xi - f'(Y_-) \eta) * \mu\right)\\
			&=  f'(Y_-) \cdot (\indicator{D_n} \cdot  Y) + ( \indicator{D_n} \xi -  \indicator{D_n} f'(Y_-) \eta) * \mu
			= (\indicator{D_n} \xi) * \mu.
	\end{align*}
	This is the second claim, and the assertion follows.
\end{proof}

\begin{example}\label{R:190423.1}
As a counterpoint to  Proposition~\ref{P:190411}, we will now exhibit an integer-valued random measure $\mu$ with the following properties.
\begin{enumerate}[label={\rm(\arabic{*})}, ref={\rm(\arabic{*})}]
\item\label{R:190423.1.i} $x\in L_\sigma(\mu)$.
\item\label{R:190423.1.ii} $x\star\mu$ is quasi-left-continuous with bounded jumps.
\item\label{R:190423.1.iii} there is a predictable process $\zeta\in L(x * (\mu-\nu))$ such that $\zeta\notin L(x\star\mu)$. 
\end{enumerate}

To this end, let $N$ denote a standard Poisson process. That is, $N$ jumps up by one with standard exponentially distributed waiting times and $B^N_t=t$ for all $t\geq 0$. 
Let now $\varphi_t = \sfrac{1}{k}$ for all $t\in [k-1,k)$ and all $k\in\N$, and fix $n \in \{1,2\}$. 
Then $\varphi^n\in L(N)\cap L(B^N)$ and $\varphi^n\cdot N = \varphi^n\cdot (N-B^N)+\varphi^n\cdot B^N$ is the sum of a uniformly integrable martingale and an increasing process (of bounded variation in the case $n = 2$).  Indeed, Kolmogorov's two-series theorem, applied to the sequence $(\varphi^n \cdot {(N-B^N)}_k)_{k \in \N}$, and an application of the Borel-Cantelli lemma, or Larsson and Ruf \cite[Corollary 4.4]{larsson.ruf.20}, yield the existence of the random variable $\varphi^n \cdot {(N-B^N)}_\infty = \lim_{t \uparrow \infty} \varphi^n \cdot {(N-B^N)}_t$.  The Burkholder-Davis-Gundy inequality yields that  $\varphi^n \cdot {(N-B^N)}$ is a uniformly integrable martingale as claimed.

In particular, it follows that the process $Y_t = \varphi^2 \cdot N_{\tan(t\wedge\sfrac{\pi}{2})}$ is  a special semimartingale on the whole time line with
$$B^Y_t = \int_0^{\tan(t\wedge\sfrac{\pi}{2})}\varphi_u^2\d u, \qquad t\geq 0.$$
	
	Statements \ref{R:190423.1.i}--\ref{R:190423.1.iii} now follow by taking $\mu = \mu^Y$, $\zeta_t = 1/\varphi_{\tan(t)}\indicator{t<\sfrac{\pi}{2}}$ and observing that $x\star\mu = Y$ and
	$$\zeta\cdot (x*(\mu-\nu))_t = \zeta\cdot(Y-B^Y)_t = {\varphi}\cdot (N-B^N)_{\tan(t\wedge\sfrac{\pi}{2})}, \qquad t \geq 0.$$
 From $\lim_{t\uparrow\infty}{\varphi}\cdot B^N_t=\infty$ we obtain 
$\zeta\notin L(B^Y)$, whereby $\zeta\in L(Y-B^Y)$ yields $\zeta\notin L(Y)=L(x\star\mu)$.\qed
\end{example}

\subsection{Sigma-locally finite variation pure-jump processes and the extended integral}	\label{SS:3.3}
The statements in  the previous subsections can also be expressed in terms of the class $\V^{\d}_\sigma$.

\begin{proposition}   \label{P:190729.2}
	If $\eta \in L_\sigma(\mu)$ then $\eta \star \mu \in \V^{\d}_\sigma$. Conversely, if $X\in \V^{\d}_\sigma$ then $x\in L_\sigma(\mu^X)$ and 
\begin{align}  \label{eq:19731.3}
	\begin{split}
		X &= X_0 + x\star\mu^X\\	
			&= X_0 + x\indicator{|x|> 1}*\mu^X+x\indicator{|x|\leq 1}*(\mu^X-\nu^X) + x\indicator{|x|\leq 1} \star \nu^X.
\end{split}
\end{align}
\end{proposition}
\begin{proof}
	The first part of the assertion
	follows directly from the definitions of $L_\sigma(\mu^X)$ and $\V^{\d}_\sigma$. The second equality in  \eqref{eq:19731.3} 
	is the consequence of Proposition~\ref{P:190330}.
\end{proof}
\begin{corollary}  \label{C:drift}
Let $X \in \V^{\d}_\sigma$. Then the following statements are equivalent. 
\begin{enumerate}[label={\rm(\roman{*})}, ref={\rm(\roman{*})}]
\item\label{C:drift.i}  $X$ is special.
\item\label{C:drift.ii} $x\in L_\sigma(\nu)$.
\end{enumerate}
Furthermore, if either condition holds then $B^X = x \star \nu^X$.
\end{corollary}
\begin{proof}
By \cite[Lemma~I.4.24]{js.03}, the semimartingale $X$ is special if and only if $x\indicator{|x|> 1}*\mu^X$ is special. By \cite[II.1.28]{js.03}, the finite variation pure-jump semimartingale $x\indicator{|x|> 1}*\mu^X$ is special if and only if $x\indicator{|x|> 1}\in L(\nu^X)$, in which case the drift of $x\indicator{|x|> 1}*\mu^X$ equals $x\indicator{|x|> 1}*\nu^X$. The claim now follows from 
Proposition~\ref{P:190729.2} and \cite[Definition~II.1.27]{js.03}.
\end{proof}
\begin{corollary}  \label{C:190730}
	Let $X \in \V^{\d}_\sigma$, $\zeta \in L(X)$, and $f$ be a twice continuously differentiable function.
Then $\zeta \cdot X, f(X) \in \V^{\d}_\sigma$ with 
	\begin{align*}
		\zeta \cdot X &= (\zeta x) \star \mu^X;\\
		f(X) &= f(X_0) + (f(X_- + x) - f(X_-)) \star \mu^X.
	\end{align*}
\end{corollary}
\begin{proof}
	This follows from Proposition~\ref{P:190729.2} in conjunction with Propositions~\ref{P:190411} and \ref{P:190528.2}.
\end{proof}

Recall that semimartingale $X$ is said to be quasi-left-continuous if $\Delta X_\tau=0$ almost surely on $\{\tau<\infty\}$ for each predictable time $\tau$. Yoeurp \cite{yoeurp.76} has shown that every local martingale can be uniquely decomposed into two components, one quasi-left-continuous and the other with jumps only at predictable times, such that the quadratic covariation of the two components is zero. 
This motivates the following result.

\begin{proposition}  \label{P:190729}
Every semimartingale $X$ has the unique decomposition 
\begin{equation}\label{eq:190520.1}
X = X_0 + X^{\q} + X^{\ddp},
\end{equation}
where  $X^\q_0=X^{\ddp}_0 = 0$, $X^{\q}$ is a quasi-left-continuous semimartingale,  $X^{\ddp}$ jumps only at predictable times, and $X^{\ddp} \in \V_\sigma^\d$. We then have $ [X^\q,X^\ddp] = 0$.
\end{proposition}

\begin{proof}
Let $\tau$ denote any predictable time. Note that $\Delta X_\tau = \Delta X^{\ddp}_\tau$ for any decomposition of $X$ by the quasi-left-continuity of $X^{\q}$. This proves the uniqueness of the decomposition.
Consider now the predictable process $(x^2\wedge 1)*\nu^X$. Applying \cite[I.2.24]{js.03} yields a family $(\tau_k)_{k \in \N}$ 
of predictable times that exhausts its jumps. Define next the bounded predictable process
\[
	\zeta=\indicator{\{\nu^X(\{\cdot\})>0\}}=\sum_{k = 1}^\infty \indicator{\lc\tau_k\rc}.
\]
Setting $X^{\q} = (1 - \zeta) \cdot X$ and $X^{\ddp} = \zeta \cdot X$ then yields the decomposition in \eqref{eq:190520.1}, the quasi-left-continuity of $X^{\q}$, and $ [X^\q,X^\ddp] = 0$.  Finally, setting $D_n =  (\Omega \times [0,\infty)) \setminus \bigcup_{k = n}^\infty \lc \tau_k\rc$ in Definition~\ref{D:190405}\ref{D:190405:i} for each $n \in \N$ yields  $X^{\ddp} \in \V_\sigma^\d$.
\end{proof}

\section{Classification of pure-jump processes}\label{S:4}
The following definition and theorem provide a precise formulation of the relationships among the various families of pure-jump processes.  For notation and setup, see Section~\ref{S:2}. For a review of the semimartingale topology on the space of semimartingales $\S$, see Appendix~\ref{S:Emery}.
\begin{definition} \label{D.1}
	Consider the following subsets of $\S$.
	\begin{itemize}
\item $\J^1$: the class of quadratic pure-jump processes, i.e., those semimartingales $X$ that satisfy 
$[X,X]^c = [X,X] - x^2 * \mu^X = 0$ (see Protter \cite[p.~63]{protter.90}).
\item $\J^2$: the class of pure-jump processes, i.e., those semimartingales $X$ that satisfy 
\begin{align} \label{eq:201027.1}
	X = X_0 + \sum_{k = 1}^\infty \Delta X_{\tau_k} \indicator{\lc \tau_k, \infty\lc}
\end{align}
	in the semimartingale topology for a family $(\tau_k)_{k \in \N}$ of stopping times. 
\item $\J^3$: the class of strong pure-jump processes, i.e., those semimartingales $X \in \J^2$ that satisfy $X = Y$ for all $Y \in \J^2$ with $\mu^Y = \mu^X$ and $Y_0 = X_0$.
\item  $\J^4 = \V^{\d}_\sigma$: the sigma-localized class of finite variation pure-jump processes.

\item   $\J^5= \V^{\d}$: the class of finite variation pure-jump processes, i.e., those semimartingales $X$ that satisfy $X=X_0+x*\mu^X$.
\item   $\J^6$: the class of piecewise constant processes with finitely many jumps on each finite time interval.\qed
\end{itemize}
\end{definition}
 
 \begin{remark}
 	Let us explain \eqref{eq:201027.1} along with the qualifier ``in the semimartingale topology.'' To this end, for each $n \in \N$, consider the process $X^{(n)}  =  X_0 + \sum_{k = 1}^n \Delta X_{\tau_k} \indicator{\lc \tau_k, \infty\lc}$. Then \eqref{eq:201027.1} should be read as $X = \lim_{n \uparrow \infty} X^{(n)}$ in the semimartingale topology, meaning that
	\begin{align*}
		\lim_{n \uparrow \infty} \left( \sup_{\zeta: |\zeta| \leq 1} \E\left[\left|  \zeta \cdot X^{(k)}_t - \zeta \cdot X_t \right| \wedge 1\right] \right) = 0  
	\end{align*}
	for all $t \geq 0$, 
	where the supremum is taken over all predictable processes $\zeta$ with $|\zeta| \leq 1$. See, in particular, Definition~\ref{D:190622}. One might consider other topologies than the semimartingale topology, for example the one induced by uniform convergence on compacts in probability. However, such a choice turns out to be impractical as illustrated in Appendix~\ref{S:6}.
\qed
 \end{remark}

\begin{theorem}\label{T:1}
We always have
\begin{align} \label{eq:190731}
	\J^1\supsetneq \J^2 \supset \J^3 \supset \J^4 \supset \J^5 \supsetneq \J^6.
\end{align}
In general, these set inclusions are strict. More precisely, there exists a filtered probability space such that simultaneously we have
\begin{align} \label{eq:190826}
	\J^2 \supsetneq \J^3 \supsetneq \J^4 \supsetneq \J^5.
\end{align}
Moreover, the following statements hold. 
\begin{enumerate}[label={\rm(\roman{*})}, ref={\rm(\roman{*})}]
	\item\label{T:1.i} For all $i \in \{1, 2, 3, 4\}$, the families  $\J^i$ equal their sigma-localized class $\J^i_\sigma$; that is, $\J^i = \J^i_{\sigma}$.  Furthermore, by definition, $\J^4 = \J^5_\sigma$.  
	\item\label{T:1.ii} For all $i \in \{1, 2,3,4,6\}$, the families  $\J^i$  are closed under stochastic integration.  
	\item\label{T:1.iii} For all $i \in \{1, \ldots, 6\}$, the families  $\J^i$  are invariant under equivalent measure changes. More precisely, with the obvious notation, if $\Qu$ is locally absolutely continuous with respect to $\P$ we have $\J^i(\P) \subset \J^i(\Qu)$ for all $i \in \{1, \ldots, 6\}$. 
\end{enumerate}
\end{theorem} 

Theorem~\ref{T:1} is proved in Section~\ref{S:5}.
The strictness of the inclusion $\J^4 \subsetneq \J^3$ is of interest. It says that there exist strong pure-jump processes, i.e., pure-jump processes uniquely determined by their jump measure, that are not sigma-locally of finite variation. To prove the strictness of this inclusion, Subsection~\ref{SS:5.6} contains a specific example of such a process $X$ (Example~\ref{ex:190828}).  This example relies on a jump measure $\mu^X$ with predictable compensator $\nu^X$ that supports a countable set of jump sizes. 

To gain insight, consider the disintegrated form $\nu^X(\d t, \d x) = F_t^X(\d x) \d A_t^X$, where $F^X$ is a transition kernel (see Section~\ref{S:2} for more details). The jump measure $\mu^X$ in Example~\ref{ex:190828} relies on a kernel $F^X$ that has large atoms in a neighbourhood of zero. As it turns out, this example is canonical. Indeed, Corollary~\ref{C:1} below states if $X$ is a strong pure-jump process whose associated jump size kernel does not allow for too many large atoms, then $X$ must be sigma-locally of finite variation.

We have already  observed that a process $X \in \J^4 \subset \J^3$ is uniquely described by its jump measure. The following corollary of the proof of Theorem~\ref{T:1} provides explicit characterizations of the processes in $\J^4$ in relation to the bigger classes $\J^1, \J^2$, and $\J^3$. A further analytic representation for such processes has been provided in Proposition~\ref{P:190729.2}.

\begin{corollary}  \label{C:1}
Let $X$ denote a process. Then the following statements are equivalent.
\begin{enumerate}[label={\rm(\roman{*})}, ref={\rm(\roman{*})}]
	\item\label{C:1.i} $X \in \J^4 = \V^{\d}_\sigma$. 
	\item\label{C:1.ii} $X \in \J^3$ and 
	\begin{align} \label{eq:190623}
		\left(\limsup_{x \downarrow 0}  x F^X(\{x\})\right)  \wedge \left(\limsup_{x \uparrow 0}  |x| F^X(\{x\})\right)  = 0, \qquad \text{$(\P \times \d A^X)$--a.e.}
	\end{align}	
	\item\label{C:1.iii} $X \in \J^2$ and 
	\begin{align} \label{eq:190828}
		\int |x| \indicator{|x| \leq 1}  F^X(\d x) < \infty, \qquad \text{$(\P \times \d A^X)$--a.e.}
	\end{align}	
	\item\label{C:1.iv} $X \in \J^1$, \eqref{eq:190828} holds, $\int_0^\cdot |\int x \indicator{|x| \leq 1} F_t^X(\d x) | \d A_t^X < \infty$, and 
	\[
		B^{X[1]} = \int_0^\cdot \left(\int x \indicator{|x| \leq 1} F_t^X(\d x) \right) \d A_t^X = x \indicator{|x| \leq 1}\star\nu^X.
	\]
	Here $B^{X[1]}$ denotes the drift of $X - x \indicator{|x| > 1} * \mu^X$; see also Section~\ref{S:2}.
\end{enumerate}
\end{corollary}
Corollary~\ref{C:1} is proved in Subsection~\ref{SS:C:1}. 	Note that the condition in \eqref{eq:190623} is satisfied, for example, if $F^X$ is atomless $(\P \times \d A^X)$--a.e.

The following is a corollary of the inclusion $\J^2 \subset \J^1$, stated in Theorem~\ref{T:1}. If $X \in \J^1$ is predictable then it is well known that $X$ is a finite-variation process. The next assertion illustrates that restricting oneself to $\J^2$ yields  a finite-variation pure-jump process.
\begin{corollary}
	If $X \in \J^2$ is predictable then $X \in \J^5 = \V^{\d}$.
\end{corollary}
\begin{proof}
	Since $X \in \J^2 \subset \J^1$ is predictable we have that $X$ is a finite-variation process.
	But then the convergence in \eqref{eq:201027.1} is actually in finite-variation and $X$ is the sum of its jumps, concluding the proof.
\end{proof}

We conclude this section with a remark on the process $X^\ddp \in \V^{\d}_\sigma$ of the semimartingale decomposition in Proposition~\ref{P:190729}. Observe that the family of predictable times $\mathcal{T}  = (\tau_k)_{k \in \N}$ from  the proof of Proposition~\ref{P:190729} exhausts the jumps of $X^\ddp$. Simultaneously, Theorem~\ref{T:1} yields $X^\ddp\in\J^2$. A priori, it is not clear that $\mathcal{T}$ is good enough to approximate $X^\ddp$ in $\J^2$ because the membership of $\J^2$ only ever guarantees one exhausting sequence of stopping times (with the desired convergence property) and that sequence is not even predictable in principle. The next result therefore appears to be rather strong.

\begin{proposition}  \label{P:190828.2}
Let $X$ satisfy $X=X^\ddp$ in the notation of Proposition~\ref{P:190729}. Let $(\tau_k)_{k \in \N}$ be any sequence of predictable times that exhausts the jumps of $X$. 
Then $(\tau_k)_{k \in \N}$ also approximates $X$ in $\J^2$, i.e., we have
$X = \sum_{k = 1}^\infty \Delta X_{\tau_k} \indicator{\lc \tau_k, \infty\lc}$	in the semimartingale topology.
\end{proposition}
\begin{proof}
Apply Lemma~\ref{L:190622.1}\ref{L:190622.1.iii} below with the same sequence $(D_n)_{n\in\N}$ as in the proof of Proposition~\ref{P:190729}.
\end{proof} 
For a related statement about the summability of jumps of a semimartingale $X=X^\ddp$ at predictable times, see Galtchouk \cite{galtchouk.80}.  There it is shown that for any sequence $(\tau_k)_{k \in \N}$ of predictable  times that exhausts the jumps of $X$ the limit $\lim_{n \uparrow \infty} \sum_{k = 1}^n \Delta X_{\tau_k} \indicator{\{\tau_k \leq t\}}$ exists almost surely for each $t \geq 0$. (Only the case when $X$ is a local martingale is considered, however the extension to the general case is straightforward.)

\section{Proof of Theorem~\ref{T:1} (and of Corollary~\ref{C:1})}  \label{S:5}

This section contains the proof of this paper's main theorem. It is split up in six subsections.  Subsections~\ref{SS:5.1}, \ref{SS:5.2}, and \ref{SS:5.3} provide the proofs of Theorem~\ref{T:1}\ref{T:1.i}, \ref{T:1.ii}, and \ref{T:1.iii}, respectively. Subsection~\ref{SS:5.4} yields the set inclusions in \eqref{eq:190731}. Then Subsection~\ref{SS:C:1}  focuses on the proof of  Corollary~\ref{C:1}, while Subsection~\ref{SS:5.6} concludes with a proof of \eqref{eq:190826}, namely the strictness of the inclusions.

\subsection{Proof of Theorem~\ref{T:1}\ref{T:1.i}}  \label{SS:5.1}
In this subsection we argue that $\J^i = \J^i_{\sigma}$ for all $i \in \{1, 2, 3, 4\}$. 
Indeed, fix $X \in \J^1_\sigma$ and the corresponding sigma-localizing 
sequence $(D_k)_{k \in \N}$ of predictable sets. Then
\[
	[X,X]^c =  \left(\lim_{k \uparrow \infty} \indicator{D_k}\right) \cdot [X,X]^c =  \lim_{k \uparrow \infty} \left( \indicator{D_k} \cdot [X,X]^c\right)  = 
	\lim_{k \uparrow \infty} \left [ \indicator{D_k} \cdot X, \indicator{D_k} \cdot X \right]^c = 0,
\]
which yields $X \in \J^1$.
As $\J^5=\V^{\d}$ is stable under sigma-stopping, \cite[Proposition~2.1]{kallsen.04} yields the statement for $i=4$.

The cases $i = 2$ and $i = 3$ follow from Lemmata~\ref{L:190623.1} and \ref{L:190623.2}. Before stating and proving them, we first present a useful tool for pure-jump processes in the next lemma.

\begin{lemma} \label{L:190622.2}
	Let $(X^{(k)})_{k \in \N} \subset \J^2$ be a sequence of pure-jump processes such that $X^{(k)}_0 = 0$ and $[X^{(k)}, X^{(l)}] = 0$ for all $k,l \in \N$ with $k \neq l$.
	Then the following two statements are equivalent. 
			\begin{enumerate}[label={\rm(\Roman{*})}, ref={\rm(\Roman{*})}] 
			\item\label{L:190622.2.I} 
			$\sum_{k=1}^\infty [X^{(k)}, X^{(k)}]   < \infty$ and $\sum_{k  = 1}^\infty B^{X^{(k)}[1]}$ converges in the  $\S$--topology to a  process $B$. 
			\item\label{L:190622.2.II} 
			$\sum_{k=1}^\infty X^{(k)}$ converges in the $\S$--topology to a process $X$.
			\end{enumerate}
			If one (hence both) of these conditions hold then $B^{X[1]} = B$, $\sum_{k = 1}^\infty \Delta X^{(k)} = \Delta X$, and $X$ is a pure-jump process.
\end{lemma}

\begin{proof} 
Thanks to Lemma~\ref{L:190729}\ref{L:190729.iii} it suffices to argue that $X$ is a pure-jump process provided the two statements hold.
For each $k \in \N$ we have a sequence $(\tau^{(k)}_n)_{n \in \N}$ of stopping times (by possibly setting $\tau^{(k)}_n = \infty$ for $n$ large enough if $X^{(k)}$ has only finitely many jumps) such that 
$X^{(k)} = \sum_{n=1}^\infty \Delta X_{\tau_n^{(k)}}^{(k)} \indicator{\lc \tau_n^{(k)}, \infty\lc}$ in the $\S$--topology and $\Delta X_{\tau_n^{(k)}}^{(k)} \neq 0$ on $\{\tau_n^{(k)} < \infty\}$.
Thanks to Lemma~\ref{L:190729}\ref{L:190729.iii} we have $\Delta X_{\tau_n^{(k)}} = \Delta X_{\tau_n^{(k)}}^{(k)}$ on $\{\tau_n^{(k)} < \infty\}$ for all $k,n \in \N$. Furthermore, $(\tau^{(k)}_n)_{k, n \in \N}$ exhausts the jumps of $X$.

Next, for each $m \in \N$, let $K_m$ and $N_m$ be the smallest integers such that
\[
	d_{\S} \left(X, \sum_{k=1}^{K_m} X^{(k)}\right) \leq \frac{1}{2 m}
\]
and 
\[
	d_{\S} \left(X^{(k)},  \sum_{n=1}^{N_m} \Delta X_{\tau_n^{(k)}}^{(k)} \indicator{\lc \tau_n^{(k)}, \infty\lc} \right) \leq \frac{1}{2 m K_m} 
	\qquad \text{for all $k \in \{1, \cdots, K_m\}$}.
\]
By a standard diagonalization argument we can now construct a sequence of stopping times $(\tau_i)_{i \in \N}$ such that 
  \[
	\lim_{m \uparrow \infty} d_{\S} \left(X, \sum_{i=1}^{m} \Delta X_{\tau_i} \indicator{\lc \tau_i, \infty\lc}\right) = 0, 
\]
yielding the statement.
\end{proof}

\begin{corollary}\label{C:190830}
The sum of two pure-jump processes whose quadratic covariation is zero is again a pure-jump process. 
\end{corollary}

The next two lemmata  exploit the fact that $\J^2$ is stable under sigma-stopping thanks to Lemma~\ref{L:190622.1}\ref{L:190622.1.ii}.

\begin{lemma}  \label{L:190623.1}
	If $X \in \S$ is sigma--locally a pure-jump process then it is a pure-jump process. 
\end{lemma}
\begin{proof}
	By assumption there exists a nondecreasing sequence $(D_k)_{k \in \N}$ of predictable sets such that $\bigcup_{k \in \N} D_k = \Omega \times [0, \infty)$ and $\indicator{D_k} \cdot X$ is a pure-jump process. With $D_0 = \emptyset$, define $\overline D_k = D_k \setminus D_{k-1}$ for all $k \in \N$ and note that $X^{(k)} = \indicator{\overline D_k} \cdot X$ is also a pure-jump process as $\J^2$ is stable under sigma-stopping.  Moreover, we have $[X^{(k)}, X^{(l)}] = 0$ for all $k,l \in \N$ with $k \neq l$ and 
	\[
		\sum_{k  = 1}^n X^{(k)} = \sum_{k  = 1}^n \left(\indicator{\overline D_k} \cdot X\right) = \indicator{D_n} \cdot X, \qquad n \in \N,
	\]
	which  converges in the $\S$--topology to $X$ (as $n \uparrow \infty$) thanks to Lemma~\ref{L:190622.1}\ref{L:190622.1.iii}.
Hence by Lemma~\ref{L:190622.2}, $X$ is a pure-jump process.
\end{proof}

\begin{lemma}  \label{L:190623.2}
	If $X \in \S$ is sigma--locally a strong pure-jump process then it is a strong pure-jump process. 
\end{lemma}
\begin{proof}
	By assumption there exists a nondecreasing sequence $(D_k)_{k \in \N}$ of predictable sets such that $\bigcup_{k \in \N} D_k = \Omega \times [0, \infty)$ and $\indicator{D_k} \cdot X$ is a strong pure-jump process. 
		Assume there exists a pure-jump process $Y$ with $\mu^Y = \mu^X$ and $Y_0 = X_0$ but $Y \neq X$. Since $\lim_{k \uparrow \infty}  (\indicator{D_k} \cdot X) = X - X_0$  and  $\lim_{k \uparrow \infty} (\indicator{D_k} \cdot Y) = Y - Y_0$ in the $\S$--topology thanks to Lemma~\ref{L:190622.1}\ref{L:190622.1.iii}, there exists some $k \in \N$ such that $\indicator{D_k} \cdot X  \neq \indicator{D_k} \cdot Y$. This, however, contradicts the assumption since $\indicator{D_k} \cdot Y \in \J^2$ for each $k \in \N$ because $\J^2$ is stable under sigma-stopping.
\end{proof}

\subsection{Proof of Theorem~\ref{T:1}\ref{T:1.ii}}   \label{SS:5.2}
In this subsection we argue that $\J^i$ is closed under stochastic integration for all $i \in \{1, 2, 3, 4, 6\}$. First, the cases $i = 1$ and $i = 6$ are clear. The case $i=4$ follows from  Corollary~\ref{C:190730}.

For the case $i=2$, assume that $X \in \J^2$ and fix $\zeta \in L(X)$. We need to argue that $\zeta \cdot X \in \J^2$. Thanks to Theorem~\ref{T:1}\ref{T:1.i} (see also Lemma~\ref{L:190623.1}), we may assume that $|\zeta|$ is bounded. The statement then follows directly from the definition of $\S$--topology.

The remaining case $i=3$ follows from the next lemma.
\begin{lemma}  \label{L:190730.1}
	 Let $X \in \J^3$ and $\zeta \in L(X)$.  If $Y \in \J^2$ is a pure-jump process with $\mu^{Y} = \mu^{\zeta \cdot X}$ then $Y = \zeta \cdot X$. 
\end{lemma}
\begin{proof}
	Note that 
	\[
		Z = \left(\indicator{\{\zeta \neq 0\}} \frac{1}{\zeta}\right)\cdot Y + \indicator{\{\zeta = 0\}} \cdot X
	\]
	 satisfies $\mu^Z = \mu^{X}$. 
	 Moreover, $Z$ is a pure-jump process thanks to Corollary~\ref{C:190830} in conjunction with the closedness of $\J^2$ under stochastic integration. Since $X \in \J^3$ we get $Z = X$.
This again yields that 
\[
	Y = \indicator{\{\zeta = 0\}} \cdot Y + \indicator{\{\zeta \neq 0\}} \cdot Y = \indicator{\{\zeta = 0\}} \cdot Y + \zeta \cdot Z =  \indicator{\{\zeta = 0\}} \cdot Y  + \zeta \cdot X.
\]
We conclude after observing that $\mu^{\indicator{\{\zeta = 0\}} \cdot Y} = 0$ and $\indicator{\{\zeta = 0\}} \cdot Y \in \J^2$, hence $\indicator{\{\zeta = 0\}} \cdot Y = 0$. The last step relies on the fact that if $X \in \J^2$ and $\mu^X = 0$ then $X = 0$; i.e., a pure-jump process that has no jumps has to equal the zero process.
\end{proof}

\subsection{Proof of Theorem~\ref{T:1}\ref{T:1.iii}}   \label{SS:5.3}
In this subsection we argue that 
if $\Qu$ is locally absolutely continuous with respect to $\P$ we have $\J^i(\P) \subset \J^i(\Qu)$ for all $i \in \{1, \ldots, 6\}$. The cases $i=1$, $i = 5$, and $i = 6$ are clear. The case $i = 2$ follows from Lemma~\ref{L:190622.1}\ref{L:190622.1.vi}. The case $i = 4$ is a consequence of Proposition~\ref{P:190729.2} and Remark~\ref{R:190405}.

The remaining case $i=3$ follows from Lemma~\ref{L:FlightToBeirut}. It requires the following result regarding the lift of a pure-jump process from $\J^2(\Qu)$ to $\J^2(\P)$.

\begin{lemma}\label{L:FlightToBeirut2}
	Let $Z$ denote a nonnegative $\P$--martingale, $\Qu$  a probability measure that satisfies $\d \Qu / \d \P|_{\sigalg F_{t }} = Z_{t}$ for all $t \geq 0$, and $Y$ an element of $\J^2(\Qu)$.
Assume that there exists some stopping time $\sigma$ such that $Y = Y^\sigma$, $\P$--almost surely, and  $Z^\sigma$ does not hit zero continuously, $\P$--almost surely,
Then there exists a $\P$--semimartingale $Y_{\uparrow} \in \J^2(\P)$ with $Y_{\uparrow} = Y$, $\Qu$--almost surely, and $Y_{\uparrow} = Y_{\uparrow}^\sigma$, $\P$--almost surely.
\end{lemma}

\begin{proof}
Let $(\tau_k)_{k \in \N}$ denote a sequence of stopping times such that $Y = Y_0 + 
\sum_{k = 1}^\infty \Delta Y_{\tau_k} \indicator{\lc \tau_k, \infty\lc}$
in the $\S$--topology under $\Qu$.
 We may now assume that $Y_0 = 0$ and $\{\tau_k < \infty\} \cap \{Z_{\tau_k} = 0\} = \emptyset$ for all $k \in \N$.  With these assumptions in place, we have 
 $[\Delta Y_{\tau_k} \indicator{\lc \tau_k, \infty\lc}, \Delta Y_{\tau_l} \indicator{\lc \tau_l, \infty\lc}] = 0$ for all $k,l \in \N$ with $k \neq l$ and  $\sum_{k=1}^\infty [\Delta Y_{\tau_k} \indicator{\lc \tau_k, \infty\lc}, \Delta Y_{\tau_k} \indicator{\lc \tau_k, \infty\lc}]   < \infty$,  $\P$--almost surely.
 Consider next the  sum $B^{(n)} = \sum_{k = 1}^n B^{\Delta Y_{\tau_k} \indicator{\lc \tau_k, \infty\lc}[1]}(\P)$  for each $n \in \N$. Suppose for the moment that $(B^{(n)})_{n \in \N}$ converges in the $\S$--topology under $\P$.  Then Lemma~\ref{L:190622.2} yields a limiting process $Y_\uparrow \in \J^2(\P)$ and Lemma~\ref{L:190622.1}\ref{L:190622.1.vi} yields that $Y_{\uparrow} = Y$, $\Qu$--almost surely. Clearly, we then also have $Y_{\uparrow} = Y_{\uparrow}^\sigma$, $\P$--almost surely. 
 
 We still need to argue the convergence of the $\P$--drifts $(B^{(n)})_{n \in \N}$ in the $\S$--topology under $\P$. To this end, we will first argue hat $\lim_{n \uparrow \infty} B^{(n)} = B$ in the $\S$--topology under $\Qu$ for some predictable, $\Qu$--almost surely  finite variation  process $B$. To see this, note that $Y \in \J^2(\Qu)$ and Lemma~\ref{L:190729}\ref{L:190729.i} in conjunction with Lemma~\ref{L:190622.1}\ref{L:190622.1.vi}  yield that
 \[
 	\left(\sum_{k = 1}^n  \Delta Y_{\tau_k} \indicator{\lc \tau_k, \infty\lc}[1] \right)_{n \in \N} \qquad \text{and} \qquad \left(\sum_{k = 1}^n  \Delta Y_{\tau_k} \indicator{\lc \tau_k, \infty\lc}[1] - B^{(n)} \right)_{n \in \N}
 \]
 both converge  in the $\S$--topology under $\Qu$. This yields the convergence of $( B^{(n)})_{n \in \N}$ in the $\S$--topology under $\Qu$ to some process $B$, which may be chosen predictable by Lemma~\ref{L:190622.1}\ref{L:190622.1.v}. By Lemma~\ref{L:190622.1}\ref{L:190622.1.iv}, we have $[B,B]^c = 0$, hence $B$ is of finite variation under $\Qu$. We also also have  $B = B^\sigma$, $\Qu$--almost surely, and may  assume that $B = B^\sigma$, $\P$--almost surely.  Similarly, we may assume that $B$ has right-continuous paths, $\P$--almost surely.
 
 Note that the first time $\sigma'$ that $B$ is of infinite variation is predictable since the total-variation process of $B$ is predictable; hence $\E_{\sigma'-}^\P[\Delta Z_{\sigma'}] = 0$ on $\{\sigma' < \infty\}$.  Since $B$ is of finite variation, $\Qu$--almost surely, we conclude that $\sigma' = \infty$, $\P$--almost surely. Hence  $B$ is of finite variation,  $\P$--almost surely, and has left limits. It suffices to argue now that the variations of $(B - B^{(n)})_{n \in \N}$ converge to zero in probability under $\P$. As they do under $\Qu$ and as the first time $\sigma''$ that the variations  of  $(B - B^{(n)})_{n \in \N}$ do not converge to zero is predictable, we may argue again that $\sigma'' = \infty$, $\P$--almost  surely. This yields  $\lim_{n \uparrow \infty} B^{(n)} = B$ in the $\S$--topology under $\P$, concluding the proof.
 \end{proof}

\begin{lemma}\label{L:FlightToBeirut}
	 Let $\Qu$ be a probability measure that is locally absolutely continuous with respect to $\P$ and let $X \in \J^3(\P)$.  If $Y \in \J^2(\Qu)$ is a pure-jump process with $\mu^{Y} = \mu^{X}$, $\Qu$--almost surely, then $Y = X$, $\Qu$--almost surely. 
\end{lemma}
\begin{proof}
For each $m \in \N$, let $\sigma_m$ be the first time that the nonnegative martingale $(\d \Qu / \d \P|_{\sigalg F_t})_{t \geq 0}$ crosses the level $1/m$. Then $X^{\sigma_m} \in \J^3(\P)$  for each $m \in \N$ by Lemma~\ref{L:190730.1} and $\lim_{m \uparrow \infty} \sigma_m = \infty$, $\Qu$--almost surely. Hence, thanks to Lemma~\ref{L:190623.2} applied under $\Qu$, we may and shall assume from now on that $X = X^{\sigma}$, where $\sigma = \sigma_m$ for some $m \in \N$.  Let $Y \in \J^2(\Qu)$ satisfy $\mu^Y = \mu^X$, $\Qu$--almost surely, and $Y_0 = X_0$. Then $Y = Y^{\sigma}$ and Lemma~\ref{L:FlightToBeirut2} yields $Y_{\uparrow} \in \J^2(\P)$ with $Y_{\uparrow} = Y$, $\Qu$--almost surely, and $Y_{\uparrow} = Y_{\uparrow}^\sigma$, $\P$--almost surely.

Define now the $\P$--semimartingale $Y' = Y_{\uparrow} + (\Delta X_\sigma - \Delta {Y_{\uparrow}}_\sigma)  \indicator{\lc \sigma, \infty\lc}$. Then we have $\mu^{Y'} = \mu^X$ and $Y'= Y$, $\Qu$--almost surely. Since $X \in \J^3(\P)$ we thus have $Y' = X$, $\P$--almost surely, yielding $Y = Y' = X$, $\Qu$--almost surely, as required.
\end{proof}

\subsection{Proof of the set inclusions in \texorpdfstring{\eqref{eq:190731}}{(1.1)}}  \label{SS:5.4}
Lemma~\ref{L:190622.1}\ref{L:190622.1.iv} yields the inclusion $\J^1 \supset \J^2$. The strictness of this inclusion follows from observing, for example, that $X_t = t$ for all $t \geq 0$ satisfies $X \in \J^1 \setminus \J^2$. The inclusions  $\J^2 \supset \J^3$ and $\J^4 \supset \J^5 \supset \J^6$ are clear.  Since the deterministic semimartingale $X = \sum_{k=1}^\infty k^{-2} \indicator{\lc 1/k, \infty\lc}$ satisfies  $X \in \J^5 \setminus \J^6$, we also have the strictness of the last inclusion.

To see $\J^3\supset \J^4$, consider now $X\in\J^5$. By definition of the $\S$--topology every exhausting sequence for $X$ also yields an approximating sequence of stopping times for $X$ in $\J^2$. This shows $\J^5\subset\J^2$, and in fact $\J^5\subset\J^3$. 
Hence Lemma~\ref{L:190623.2} yields $\J^4 = \J^5_\sigma\subset \J^3_\sigma = \J^3$.

\subsection{Proof of Corollary~\texorpdfstring{\ref{C:1}}{1.3} 
and the strictness of the inclusion \texorpdfstring{$\J^2 \supset \J^3$}{J2 supset J3}} \label{SS:C:1}

Let us outline here the proof.
\begin{itemize}
	\item Assume that Corollary~\ref{C:1}\ref{C:1.i} holds. Then Proposition~\ref{P:190729.2} yields that 
$x\in L_\sigma(\mu^X)$. Thanks to Proposition~\ref{P:190330} we moreover have $x \indicator{|x| < 1} \in L_\sigma(\nu^X)$.  Lemma~\ref{L:Kallsen}, the set inclusions of \eqref{eq:190731}, proven in Subsection~\ref{SS:5.4}, and the representation in \eqref{eq:19731.3} yield that Corollary~\ref{C:1}\ref{C:1.ii}--\ref{C:1.iv} also hold. 
	\item  Assume that Corollary~\ref{C:1}\ref{C:1.iv} holds.  Thanks to Proposition~\ref{P:190330}
	and Lemma~\ref{L:Kallsen} we then have $x \in L_\sigma(\mu^X)$ and
	\[
		x \star \mu^X = x \indicator{|x|>1} * \mu^X + x \indicator{|x| \leq 1} * (\mu^X - \nu^X) +  x \indicator{|x|\leq 1} * \nu^X.
	\]
	However, the right hand side of the last display is exactly $X-X_0$, thanks to the representation in  \cite[II.2.34]{js.03} for quadratic pure-jump processes $X$, namely
\begin{align} \label{eq:190827.1}
	X = X_0 + x \indicator{|x|>1} * \mu^X + x \indicator{|x| \leq 1} * (\mu^X - \nu^X) + B^{X[1]}.
\end{align} 
	Hence $X - X_0 = x \star \mu^X$ and Proposition~\ref{P:190729.2} then yields
	 the implication from \ref{C:1.iv} to \ref{C:1.i} in Corollary~\ref{C:1}.
\item Assume that Corollary~\ref{C:1}\ref{C:1.iii} holds.  Then Lemma~\ref{L:190623.3}\ref{L:190623.3.i} below yields that Corollary~\ref{C:1}\ref{C:1.i} also holds.
\item Assume that Corollary~\ref{C:1}\ref{C:1.ii} holds.  Then Lemma~\ref{L:190827.1} below yields that Corollary~\ref{C:1}\ref{C:1.i} also holds.
\item Finally, Lemma~\ref{L:190827.2} shows that the inclusion $\J^2 \supset \J^3$ is usually strict.
\end{itemize}

\begin{lemma} \label{L:190623.3}
	Let $X \in \J^2$ be a pure-jump semimartingale and define the predictable set
	\[
		D =  \left\{( \omega, t): \int |x|  \indicator{|x| \leq 1} F_t^X(\d x) < \infty   \right\}.
	\]
	Then the following statements hold.
	\begin{enumerate}[label={\rm(\roman{*})}, ref={\rm(\roman{*})}] 
	\item \label{L:190623.3.i}  We have $\indicator{D} \cdot X \in \V^{\d}_\sigma$. 
	\item \label{L:190623.3.ii} The following equalities hold.
	\begin{align*}
		D &{}= \left\{( \omega, t): \int x^+  \indicator{x^+ \leq 1}  F_t^X(\d x) < \infty   \right\} \\
		  &{}= \left\{( \omega, t): \int x^-  \indicator{x^- \leq 1}  F_t^X(\d x) < \infty   \right\} ,\qquad \text{$(\P \times \d A^X)$--a.e.}
	\end{align*}
		\item \label{L:190623.3.iii} There exists a predictable process $\beta^X$ with $ |\beta^X| \cdot A^X< \infty$ such that $B^{X[1]} =  \beta^X \cdot A^X$.  Moreover, on $D$ we have $\beta^X = \int x \indicator{|x| \leq 1}  F^X(\d x)$, $(\P \times \d A^X)$--a.e.
	\end{enumerate}
\end{lemma}

\begin{proof}
	For \ref{L:190623.3.i}, thanks to Theorem~\ref{T:1}\ref{T:1.ii} (or the fact that $\J^2$ is stable under sigma-stopping), we may  assume that $X = \indicator{D} \cdot X$. Define now 
		 the predictable sets
		\[
		 D_k = \left\{( \omega, t):\int |x| \indicator{|x| \leq 1}  F_t^X(\d x) \leq k  \right\}, \qquad k \in \N,
	\]
	and   $\overline D_k = D_k \setminus D_{k-1}$ with $D_0 = \emptyset$. Again  $X^{(k)} =  \indicator{\overline D_k} \cdot X$ is a pure-jump process for each $k \in \N$.  
	Moreover, by Proposition~\ref{P:190330} and Lemma~\ref{L:Kallsen}, $x \in L_\sigma(\mu^{X^{(k)}})$ for each $k \in \N$. Since by \eqref{eq:190731} $x \star \mu^{X^{(k)}} \in \J^4 \subset \J^3$ we have $X^{(k)} = x \star \mu^{X^{(k)}}$ for each $k \in \N$. 
Thanks to  Proposition~\ref{P:190729.2} this yields 
	\[
		B^{X^{(k)}[1]} =  \int_0^\cdot \indicator{\overline D_k}(t) \left(\int x \indicator{|x| \leq 1}  F_t^X(\d x) \right) \d A^X_t, \qquad k \in \N.
	\]
	Hence by Lemmata~\ref{L:190622.1}\ref{L:190622.1.iii} and \ref{L:190622.2} we have
	\[
		 \int_0^\cdot \indicator{D_n}(t) \left(\int x \indicator{|x| \leq 1}   F_t^X(\d x) \right) \d A^X_t
	\]
	converges in the $\S$--topology (as $n \uparrow \infty$) to a finite variation process, yielding the result.

	To see \ref{L:190623.3.ii},
		fix $\kappa > 0$  and consider the predictable set
	\[
		 D' =  \left\{( \omega, t): \int x^-  \indicator{x^- \leq 1}   F_t^X(\d x) < \kappa  \right\}.
	\]	
	By symmetry, it suffices now to argue that $ \indicator{D'} x^+  \indicator{x^+ \leq 1}   \nu^X < \infty$.  As above,  we may assume that $X = \indicator{D'} \cdot X$. Then $x^- * \mu^X < \infty$ and \eqref{eq:190731.2} with $\zeta = 1$ yield that also $x^+ * \mu^X < \infty$; hence $x^+  \indicator{x^+ \leq 1} * \nu^X < \infty$ as required.

	For \ref{L:190623.3.iii} 
	Lebesgue's  decomposition of measures yields a predictable set $\widetilde D$ and a predictable process $\beta^X$ such that   $\indicator{\widetilde D} \cdot A^X = 0$, $|\beta^X| \cdot A^X< \infty$, and $B^{X[1]} =   \indicator{\widetilde D} \cdot B^{X[1]} + \beta^X \cdot A^X$. 
	Next, note that $\indicator{\widetilde D} \cdot X$ is also a pure-jump process and $x^2   * \nu^{\indicator{\widetilde D} \cdot X[1]} =  \indicator{\widetilde D} x^2   * \nu^{X[1]} = 0,$ yielding $\indicator{\widetilde D} \cdot X[1] = 0$, hence $ \indicator{\widetilde D} \cdot B^{X[1]}  = 0$.  The second assertion of \ref{L:190623.3.iii} follows directly from \ref{L:190623.3.i} and Proposition~\ref{P:190729.2}.
\end{proof}

\begin{lemma}  \label{L:190827}
	Let $\mu$ be a jump measure with $x^2 * \mu < \infty$. Moreover, let $(f^{(k)})_{k \in \N}$ and $(g^{(k)})_{k \in \N}$ denote two nonincreasing sequences of strictly positive predictable processes. For each $k \in \N$, define 
	\[
		\beta^{(k)}  = \int x \indicator{\{x \in [-1, -g^{(k)}] \cup [f^{(k)}, 1]\} } F(\d x).
	\]
	If $|\beta^{(k)}| \cdot A < \infty$ for each $k \in \N$ and $\beta^{(k)} \cdot A$ converges in the $\S$--topology (as $k \uparrow \infty$) to a process $B$, then there exists a pure-jump process $X$ with $\mu^X = \mu$ and $B^{X[1]} = B$. 
\end{lemma}
\begin{proof}
	Note that
	\[
		 \int |x| \indicator{\{x \in [-1, -g^{(k)}] \cup [f^{(k)}, 1] \}} F(\d x) < \infty, \qquad \text{$(\P \times \d A)$--a.e.}, \qquad k \in \N,
	\]
	by the strict positivity of $g^{(k)}$ and $f^{(k)}$. 
	Hence by assumption and Propositions~\ref{P:190330} and \ref{P:190729.2}, 
	$x \indicator{|x| = 1}, x \indicator{\{x \in (-g^{(k-1)}, -g^{(k)}] \cup [f^{(k)}, f^{(k-1)}) \}} \in L_\sigma(\mu)$, $X^{(0)} = x \indicator{|x| = 1} \star \mu \in \V^{\d}_\sigma \subset \J^2$,  and
	\[
		X^{(k)} = x \indicator{\{x \in (-g^{(k-1)}, -g^{(k)}] \cup [f^{(k)}, f^{(k-1)}) \}} \star \mu \in \V^{\d}_\sigma \subset \J^2, \qquad k \in \N,
	\]
	where $g^{(0)} = f^{(0)} = 1$. 
	An application of Lemma~\ref{L:190622.2} now concludes.
\end{proof}

The following lemma complements Lemma~\ref{L:190623.3}\ref{L:190623.3.ii}\&\ref{L:190623.3.iii}.  Given a jump measure $\mu$ and a predictable process $\beta$, both satisfying technical conditions, it constructs a pure-jump process $X$ with $\mu^X = \mu$ and drift rate $\beta$ on the predictable set where the jump sizes do not integrate.

\begin{lemma} \label{L:190623.4}
	Let $\mu$ be a jump measure with $x^2 * \mu < \infty$ and
	\begin{align}   \label{eq:190827.3} 
		\left(\limsup_{x \downarrow 0}  x F(\{x\})\right)  \wedge \left(\limsup_{x \uparrow 0}  |x| F(\{x\})\right)  = 0, \qquad \text{$(\P \times \d A)$--a.e.}
	\end{align}	
	Assume that the predictable set
	\[
		D =  \left\{( \omega, t): \int |x|  \indicator{|x| \leq 1} F_t(\d x) < \infty   \right\}
	\]
	satisfies
	\begin{equation}\label{eq:190827.2}
	\begin{split} 
			D &{}= \left\{( \omega, t): \int x^+  \indicator{x^+ \leq 1}  F_t(\d x) < \infty   \right\} \\
			  &{}= \left\{( \omega, t): \int x^-  \indicator{x^- \leq 1}  F_t(\d x) < \infty   \right\} , \qquad \text{$(\P \times \d A)$--a.e.}
	\end{split}
	\end{equation}
	and let $\beta$ denote any nonnegative predictable process such that
	\[
		\int_0^\cdot \indicator{D^c}(t) |\beta_t| \d A_t < \infty.
	\]
	 Then there exists a pure-jump process $X \in \J^2$ such that $\mu^X = \mu$ and
	\[
		B^{X[1]} =  \int_0^\cdot \left(\indicator{D}(t) \int x \indicator{|x| \leq 1}  F_t (\d x) + \indicator{D^c}(t) \beta_t\right) \d A_t.
	\]
\end{lemma}

\begin{proof}
Consider the predictable sets 
\[
	D' = D^c \cap \left\{( \omega, t): \limsup_{x \downarrow 0}  x F_t(\{x\}) = 0\right\};\qquad D'' = D^c\setminus D'.
\]
By Corollary~\ref{C:190830}, symmetry, and Subsection~\ref{SS:3.3} we may assume that $D''=\emptyset$, $\indicator{|x| > 1} * \mu = 0$, $\mu = \indicator{D'}\mu$, and $\beta = \indicator{D'} \beta$.  

To make headway, consider the predictable process
	\begin{align*}
		c &= \inf \left\{\varepsilon > 0:  \varepsilon F\left(\left\{\varepsilon\right\}\right) > 1 \right\} \wedge 2
	\end{align*}
	and note that $c > 0$ by assumption.  Next, consider the process
	\begin{align*}
		d&= \indicator{D'^c} + \indicator{D'} \sup \left\{\varepsilon > 0: \beta +  \int  |x|  \indicator{x \in [-1, -\varepsilon]}  F(\d x) \geq \int x  \indicator{\{x \in [c, 1]\}}  F(\d x)  \right\}.
	\end{align*}	
		Then by \eqref{eq:190827.2}, $d>0$ and since 
	 \[
	 	\left\{(\omega, t):  d_t \geq \varepsilon\right\} =   \left\{(\omega, t): \int  |x|  \indicator{x \in [-1, -\varepsilon]}  F_t(\d x) \geq  \int x  \indicator{\{x \in [c, 1]\}}  F_t(\d x)   - \beta\right\} \in \Pcal, 
	 \]
	 for all $\varepsilon \in (0,1]$, it is easy to see that $d$ is predictable.
	 Next, define the processes
	\begin{align*}
		g^{(k)} &= d \wedge \frac{1}{k}, \qquad k \in \N;\\
		f^{(k)} &= \indicator{D'^c} + \indicator{D'} \sup \left\{\varepsilon > 0: \int x  \indicator{x \in [\varepsilon, 1]}  F(\d x) \geq \beta + \int |x| \indicator{\{x \in [-1, -g^{(k)}]\}} F(\d x) \right\}, \, k \in \N.
	\end{align*}
	Similarly as for $d$ we may argue that $f^{(k)} > 0$ and $f^{(k)}$ is predictable for each $k \in \N$. 
Again by \eqref{eq:190827.2}, we also have $\lim_{k \uparrow \infty} f^{(k)} = 0$ on $D'$.  Note that we also have $f^{(k)} < c$ for each $k \in \N$.
	 
	 Next, define
	 	\[
		\beta^{(k)}  = \int x \indicator{\{x \in [-1, -g^{(k)}] \cup [f^{(k)}, 1] \}} F(\d x), \qquad k \in \N.
	\]
	Since 	
	\[
		\beta^{(k)} \in \left[\beta, \beta +  f^{(k)} F\left(\left\{f^{(k)}\right\}\right)\right] \subset [\beta, \beta + 1], \qquad k \in \N,
	\]
	we have $|\beta^{(k)}| \cdot A < \infty$ and
	\[
		\int_0^\cdot \beta^{(k)} \d A_t \in \left[\int_0^\cdot \beta_t \d A_t,  \int_0^\cdot \left( \beta_t + f^{(k)}_t F_t\left(\left\{f^{(k)}_t\right\}\right)  \right)  \d A_t  \right],
		\qquad k \in \N.
	\]

	Since $\lim_{k \uparrow \infty} f^{(k)} = 0$ on $D'$ we also have 
	\[
		\lim_{k \uparrow \infty}  f^{(k)} F\left(\left\{f^{(k)}\right\}\right) = 0
	\] 
	by assumption. Hence,  dominated convergence yields $\lim_{k \uparrow \infty} \beta^{(k)} \cdot A  = \beta \cdot A$ in the $\S$--topology.  
	An application of Lemma~\ref{L:190827} now concludes.
\end{proof}

\begin{lemma} \label{L:190827.1}
	Let $X \in \J^3$ denote a strong pure-jump process such that \eqref{eq:190623} holds. Then $X \in \J^4$. 
\end{lemma}
\begin{proof}
	We only need to argue that \eqref{eq:190828} holds.  Recall Lemma~\ref{L:190623.3}\ref{L:190623.3.ii} and apply Lemma~\ref{L:190623.4} with $\mu = \mu^X$ and $\widetilde \beta = 1$ and $\overline \beta = -1$. If  \eqref{eq:190828} did not hold then we would obtain two pure-jump processes $\widetilde X$ and $\overline X$ with $\widetilde X \neq \overline X$ but with the same jump measures, contradicting the fact that  $X \in \J^3$. This concludes the proof.
\end{proof}

\begin{lemma} \label{L:190827.2}
	Assume that the filtered probability space is large enough so that it supports a probability measure $\mu$ that satisfies \eqref{eq:190827.3}, \eqref{eq:190827.2}, and
	\begin{align*}
		\P\left[\int_0^\cdot \indicator{\{\int |x| \indicator{|x| \leq 1}  F_t(\d x) = \infty\}} \d A_t > 0\right] > 0.
	\end{align*}	
	  Then $\J^2 \neq \J^3$.
\end{lemma}
\begin{proof}
	As in the proof of Lemma~\ref{L:190827.1}, consider the two predictable processes $\widetilde \beta = 1$ and $\overline \beta = -1$ and conclude by applying Lemma~\ref{L:190623.4} twice.
\end{proof}

As an illustration of Lemma~\ref{L:190827.2} and a preparation for the next subsection, let us now discuss the L\'evy situation by means of the following example. When $X$ is a L\'evy process we abuse the notation to treat $F^X$ as a deterministic measure over $\R$ rather than a stochastic process. 
\begin{example}
 Let $X$ be an $\alpha$--stable L\'evy process without Brownian component. Specifically, take $F^X(\d x) = \indicator{x \neq 0} |x|^{-1-\alpha} \d x$ for all $x \in \R$ with $0<\alpha<2$ and $A_t^X = t$ for all $t \geq 0$. Observe that  $X\in\J^4$ is equivalent to $X\in\J^5$ since $X$ is L\'evy. Let us now write $\beta = B^{X[1]}_t/t$, where $t > 0$, for the drift rate of $X$.
\begin{itemize}
\item  If $0<\alpha<1$ and $\beta = \int x \indicator{|x| \leq 1}  F^X(\d x)=0$  then $X$ belongs to $\J^5\setminus\J^6$.
\item  If $0<\alpha<1$ and $\beta  \neq \int x \indicator{|x| \leq 1}  F^X(\d x)$  then $X$ belongs to $\J^1\setminus\J^2$. 
\item  If $1\leq\alpha<2$ then $X$ belongs to $\J^2\setminus\J^3$ for any value of $\beta$.\qed
\end{itemize}
\end{example}

\subsection{Proof of \texorpdfstring{\eqref{eq:190826}}{(1.2)}}   \label{SS:5.6}
On finite probability spaces we have $\J^2 = \J^5$. However, in general, this is not true. Lemma~\ref{L:190827.2} already asserts that $\J^2 \neq \J^3$ as long as the probability space is large enough.  The process $X$ of the introduction shows that usually we have $\J^4 \neq \J^5$. Example~\ref{ex:190828} below illustrates that 
$\J^3 \neq \J^4$ is also possible.

Theorem~\ref{T:1} asserts that all these inequalities may hold simultaneously for some probability space. To see that such a probability space exists it suffices to piece together these three examples. For example, take the product of a probability space that allows for a process as in Example~\ref{ex:190828} and another probability space that satisfies the assumptions of Lemma~\ref{L:190827.2} and additionally allows for a process as in the introduction.
As filtration consider the one of Example~\ref{ex:190828} between time 0 and 1 and afterwards allow the filtration to be large enough to allow for the other examples.

Example~\ref{ex:190828} requires a few technical prerequisites that we introduce now. Throughout this subsection, $\filt{F}^X$ and $\filt{F}^X_+$ shall denote the natural filtration of a process $X$ and   its right-continuous modification.  Additionally, for purely technical reasons, throughout this subsection we also assume to be on the canonical space as in Dellacherie and Meyer \cite[Definition~IV.95]{dellacherie.meyer.78}.

\begin{lemma} \label{L:190525.1}
	Let $g$ denote a $\{0,1\}$--valued $\filt{F}^X$--predictable function and $X \in \S$ a semimartingale. Then $g(\Delta X)$ is $\filt{F}^X$--optional. Moreover, if $\tau$ is an $\filt{F}^X$--stopping time then $\lc \tau, \infty \rc = \{g(\Delta X)=1\}$ for some $\{0,1\}$--valued $\filt{F}^X$--predictable function $g$.
\end{lemma}
\begin{proof}
First, $g(\Delta X)$ is optional as a composition of appropriately measurable functions. Let $\Ocal^X$ now denote the $\filt{F}^X$--optional sigma algebra.  It suffices to prove that $\Ocal^X \subset \overline \Ocal$, where $\overline \Ocal = \bigcup_g \{g(\Delta X)=1\}$ with the union is taken over all $\{0,1\}$--valued $\filt{F}^X$--predictable functions. First note that $\overline \Ocal$ is a sigma algebra since the maximum of countably many $\{0, 1\}$--valued predictable functions is again predictable. Next, taking $g = \indicator{E}$, with a slight misuse of notation, for $E$ either an event in the $\filt{F}^X$--predictable sigma algebra or in the Borel sigma algebra on $\R$, shows that $\overline \Ocal$ contains the $\filt{F}^X$--predictable sigma algebra and the one generated by $\Delta X$.  Since the $\filt{F}^X$--predictable sigma algebra together with the sigma algebra generated by $\Delta X$ generates $\Ocal^X$ (see \cite[Theorem~IV.97a]{dellacherie.meyer.78}), we have indeed the inclusion $\Ocal^X \subset \overline \Ocal$ and the statement follows.
\end{proof}

\begin{lemma} \label{L:190525.2}
	Assume that $X$ is a L\'evy process.  For an $\filt{F}^X_+$--stopping time $\tau$ we then have $\lc \tau \rc = \{g(\Delta X)=1\}$ for some $\{0,1\}$--valued $\filt{F}^X$--predictable function $g$.
\end{lemma}
\begin{proof}
Since $X$ is Feller, $\filt{F}^X_+$ can be obtained from $\filt{F}^X$ by augmenting it with the null sets of $\sigalg{F}^X_\infty$. Hence there exists an  $\filt{F}^X$--stopping time $\overline \tau$ with $\overline \tau = \tau$, $\P$--almost surely. Thus, an application of Lemma~\ref{L:190525.1} concludes.
\end{proof}

\begin{lemma} \label{L:190626}
	Assume that $X$ is a L\'evy process  with $|\Delta X| \leq 1$ and assume that Y is an $\filt{F}^X_+$--pure-jump process 
	with $\mu^Y = \mu^X$.  Then 
	there exists a nondecreasing sequence $(f^{(n)})_{n \in \N}$ of $\{0,1\}$--valued $\filt{F}^X$--predictable functions such that 
	\begin{align} 
		\int_0^\cdot \left( \int \left(|x| f^{(n)}_t(x)\right) F^X(\d x) \right) \d t &< \infty, \qquad n \in \N; \label{eq:190627.1a} \\
		\lim_{n \uparrow \infty} \int_0^\cdot \left(\int \left(x f^{(n)}_t(x)\right) F^X(\d x) \right) \d t &= B^Y \qquad \text{in the $\S$--topology.} \label{eq:190627.1b}
	\end{align}
\end{lemma}
\begin{proof}
Let $(\tau_k)_{k \in \N}$ be an exhausting sequence of $\filt{F}^X_+$--stopping times for the jumps of $Y$ such that $Y = \sum_{k=1}^{\infty} \Delta X_{\tau_k} \indicator{\lc \tau_k, \infty\lc}$
in the $\S$--topology.  
By Lemma~\ref{L:190525.2}, there exists an $\filt{F}^X$--predictable $\{0,1\}$--valued function ${g^{(k)}}$ such that
\[
	\Delta X_{\tau_k} \indicator{\lc \tau_k, \infty\lc} = x {g^{(k)}} (x) * \mu^X, \qquad k \in \N.
\]
 Observe also that 
$$
 \int_0^\cdot \left( \int \left(|x| g^{(k)}_t(x)\right) F^X(\d x) \right) \d t
= B^{|\Delta X_{\tau_k}| \indicator{\lc \tau_k, \infty\lc}} < \infty, \qquad k\in\N.
$$

Since the  elements of $(\tau_k)_{k \in \N}$ have disjoint support we may assume that 
$f^{(n)} = \sum_{k = 1}^n {g^{(k)}}$ is also $\{0,1\}$--valued for each $n \in \N$. Then clearly \eqref{eq:190627.1a} holds and Lemma~\ref{L:190622.2} yields that
\begin{align*}
	\lim_{n \uparrow \infty} \left(\int_0^\cdot \left(\int x f^{(n)}_t(x) F^X(\d x) \right) \d t\right) &= 
	\lim_{n \uparrow \infty}\sum_{k  = 1}^n B^{\Delta X_{\tau_k} \indicator{\lc \tau_k, \infty\lc}} = B^Y
\end{align*}
in the $\S$--topology, yielding \eqref{eq:190627.1b}.
 \end{proof}

\begin{example} \label{ex:190828}
	Let $X$ be a L\'evy process with L\'evy measure 
	\[
		F^X(\d x) = \sum_{k = 1}^\infty k^2 3^{2k} \delta_{1/(k^2 3^k)}(\d x) + \sum_{k = 1}^\infty k^2 3^{2k} \delta_{-1/(k^2 3^k)}(\d x), \qquad x \in \R,
	\]
	and without a drift and Brownian motion component; in particular, $B^{X[1]} = 0$.
	Note that 
	\[
		\int (x^2\wedge |x|) F^X(\d x) = 2  \sum_{k = 1}^\infty k^2 3^{2k} \frac{1}{k^4 3^{2k}} = 2 \sum_{k = 1}^\infty \frac{1}{k^2} < \infty.
	\]
	Moreover, $X$ is a pure-jump process by Lemma~\ref{L:190622.2} with $X^{(k)} = x \indicator{|x| \in (1/(k+1),1/k]} * \mu^X$ for each $k \in \N$. 
	
	Since $\int |x|F^X(\d x) = \infty$ it is clear that $X \notin V_\sigma^d$. However, we claim that $X$ is a strong pure-jump process, i.e., $X \in \J^3$.  Indeed, let $Y$ denote any pure-jump process with $\mu^Y = \mu^X$. Thanks to the canonical  representation of quadratic pure-jump processes in \eqref{eq:190827.1} it suffices to show that $B^Y = 0$.  

	To this end, thanks to Lemma~\ref{L:190626}, there exists a nondecreasing sequence 
	$(f^{(n)})_{n \in \N}$ of\linebreak $\{0,1\}$--valued $\filt{F}^X$--predictable functions such that \eqref{eq:190627.1a} and \eqref{eq:190627.1b} hold.  Assume that $B^Y~\neq~0$. By Lemma~\ref{L:190623.3}\ref{L:190623.3.iii} there exist some $\kappa \in \N$ and a predictable set $D$ such that $\int_0^\cdot \indicator D(t)  |\d B_t^Y| > 0$ and 
	\begin{align} \label{eq:190627.2}
		\indicator{D}  \sup_{n \in \N} \left|  \int \left(x f^{(n)}(x)\right) F^X(\d x) \right| < \frac{3^{\kappa}}{2}, \qquad \text{$(\P \times \d t)$--a.e.} 
	\end{align}
	Consider now the predictable sets 
	\begin{align*}
		\mathcal A^{(n), +}_t       &{}= \left \{k \in \N: f_t^{(n)}\left(\frac{1}{k^2 3^k}\right) = 1   \right\},\qquad t \geq 0,\, n \in \N; \\
		\quad \mathcal A^{(n), -}_t &{}= \left \{k \in \N: f_t^{(n)}\left(-\frac{1}{k^2 3^k}\right) = 1   \right\}, \qquad t \geq 0,\, n \in \N,
	\end{align*}
	along with their symmetric differences 
	\[
		\mathcal A_t^{(n)}  = \left(\mathcal A^{(n), +}_t \setminus \mathcal A^{(n), -}_t\right) \cup \left(\mathcal A^{(n), -}_t \setminus \mathcal A^{(n), +}_t\right), \qquad t \geq 0,\, n \in \N.
	\]
	Thanks to \eqref{eq:190627.1a}, $k^{(n)}_t = \max \mathcal A_t^{(n)} < \infty$ (with $\max \emptyset = 0$), $(\P \times \d t)$--a.e., for each $n \in \N$.  If $k^{(n)}_t = 0$ then  $\int (x f_t^{(n)}(x)) F^X(\d x) = 0$ for all $t \geq 0$ and $n \in \N$.   If $k^{(n)}_t  \in \N$ then 
	\[
		\left|\int \left(x f_t^{(n)}(x)\right) F^X(\d x) \right| \geq 3^{k^{(n)}_t} - \sum_{k = 1}^{k^{(n)}_t - 1} 3^k \geq 3^{k^{(n)}_t} - \frac{ 3^{k^{(n)}_t}}{2} =   \frac{ 3^{k^{(n)}_t}}{2}, \qquad t \geq 0,\, n \in \N.
	\]
 	Hence by \eqref{eq:190627.2}, on $D$, we have $\max \mathcal A^{(n)} < \kappa$ for all $n \in \N$. 
	We have just argued that 
	\[
		\indicator{D} \int \left(\indicator{|x| \leq 1/(\kappa^2 3^\kappa)}  x f^{(n)} (x)  \right) F^X(\d x) = 0, \qquad \text{$(\P \times \d t)$--a.e.},\, n \in \N.
	\]
	Thus 
	\[
		\indicator{D} \cdot B^Y = \lim_{n \uparrow \infty} \int_0^\cdot \left( \indicator{D}(t) \int \left(\indicator{|x| > 1/(\kappa^2 3^\kappa)}  x f^{(n)}_t (x)  \right) F^X(\d x)  \right) \d t = 0
	\]
	in the $\S$--topology since $\indicator{D} \indicator{|x| > 1/(\kappa^2 3^\kappa)} x * \mu^X \in \V^{\d}$, hence a strong pure-jump process. This is a contradiction to the assumption that $\int_0^\cdot \indicator D(t)  |\d B_t^Y| > 0$.  This shows that $X$ is a strong pure-jump process. \qed
	\end{example}

\appendix
\section{\'Emery's semimartingale topology}  \label{S:Emery}

Here we  briefly review the definition and basic facts of the semimartingale topology (in short, $\S$--topology), 
introduced by \'Emery \cite{emery.79}. 

\begin{definition} \label{D:190622}
	Let $(X^{(k)})_{k \in \N} \subset \S$ denote a sequence of semimartingales. We say that this sequence converges to $X \in \S$ in the semimartingale topology (in short, $\S$--topology) if 
		\begin{align} \label{eq:190731.2}
		\lim_{k \uparrow \infty} \left( \sup_{\zeta: |\zeta| \leq 1} \E\left[\left| \zeta_0 X^{(k)}_0 + \zeta \cdot X^{(k)}_t - \zeta_0 X_0 - \zeta \cdot X_t \right| \wedge 1\right] \right) = 0  
	\end{align}
	for all $t \geq 0$, 
	where the supremum is taken over all predictable processes $\zeta$ with $|\zeta| \leq 1$.  \qed
\end{definition}

The space $\S$ equipped with this topology is a complete metric space \cite[Theoreme~1]{emery.79}, say with distance $d_\S$. Note that if a sequence $(X^{(k)})_{k \in \N} \subset \S$ converges in the $\S$--topology it also converges in the sense of uniform convergence on compacts in probability.

\begin{remark} \label{R:190729}
	In contrast to \'Emery \cite{emery.79}, we have not assumed (nor excluded) that the underlying filtration $\filt{F}$ be augmented by the $\P$--null~sets. Nevertheless, the cited results of 	\cite{emery.79}  below can be applied by choosing appropriate process modifications. For example, $\S$ equipped with the $\S$--topology is a complete metric space as any limit (in the augmented filtration) can be identified with an  $\filt{F}$--semimartingale by taking appropriate modifications. See, for example, Perkowski and Ruf \cite[Appendix~A]{Perkowski_Ruf_2014} for a summary of these techniques.
	\qed
\end{remark}

We now collect some well known facts concerning the $\S$--topology.
\begin{lemma}  \label{L:190622.1}
	Let $(X^{(k)})_{k \in \N} \subset \S$ denote a sequence of semimartingales with $X^{(k)}_0  = 0$. Then the following statements hold.
	\begin{enumerate}[label={\rm(\roman{*})}, ref={\rm(\roman{*})}] 
		\item\label{L:190622.1.i} If the sequence $(X^{(k)})_{k \in \N}$ converges locally in the $\S$--topology then it also converges  in the $\S$--topology.
		\item\label{L:190622.1.ii} If $\lim_{k \uparrow \infty} X^{(k)} = X$ in the $\S$--topology for some semimartingale $X \in \S$ and if $D$ is a predictable set then $\lim_{k \uparrow \infty} (\indicator{D} \cdot X^{(k)}) = \indicator{D} \cdot X$ in the $\S$--topology.
		 \item\label{L:190622.1.iii} If $(D_k)_{k \in \N}$ is a nondecreasing sequence of predictable sets  such that $\bigcup_{k \in \N} D_k = \Omega \times [0, \infty)$ and $X$ is a semimartingale with $X_0 =0$ then $\lim_{k \uparrow \infty} (\indicator{D_k} \cdot X) = X$ in the $\S$--topology.
		\item\label{L:190622.1.iv} If $\lim_{k \uparrow \infty} X^{(k)} = X$ in the $\S$--topology for some semimartingale $X \in \S$ we have $$\lim_{k \uparrow \infty}  [X^{(k)}, X^{(k)}]  = [X, X] \qquad\text{and} \qquad
		\lim_{k \uparrow \infty}  [X^{(k)}, X^{(k)}]^c  = [X, X]^c$$ in the $\S$--topology.
		\item\label{L:190622.1.v} If $\lim_{k \uparrow \infty} X^{(k)} = X$ in the $\S$--topology for some semimartingale $X \in \S$ and if $X^{(k)}$ is predictable for each $k \in \N$ then $X$ has a predictable version.
		\item\label{L:190622.1.vi} Assume that the probability measure $\Qu$ is locally absolutely continuous with respect to $\P$. If $\lim_{k \uparrow \infty} X^{(k)} = X$ in the $\S$--topology for some semimartingale $X \in \S$ under $\P$ then also $\lim_{k \uparrow \infty} X^{(k)} = X$ in the $\S$--topology  under $\Qu$.
	\end{enumerate}
\end{lemma}
\begin{proof}
	First, \ref{L:190622.1.i} and \ref{L:190622.1.ii} follow from the definition of $\S$--topology and \ref{L:190622.1.iii} and \ref{L:190622.1.iv} are argued in \cite[Proposition~3 and Remarque~1 on p.~276]{emery.79}. To see \ref{L:190622.1.v}, recall that also $\lim_{k \uparrow \infty} X^{(k)} = X$ (in the sense of uniform convergence on compacts); hence also almost surely along a subsequence. In conjunction with Remark~\ref{R:190729} this yields the claim.
	Finally, \ref{L:190622.1.vi} is proved by applying \cite[Proposition~6]{emery.79} in conjunction with \ref{L:190622.1.i}.  
\end{proof}

Next, we consider sums of semimartingales and their convergence in  the $\S$--topology.

\begin{lemma}\label{L:190729}
Let $(X^{(k)})_{k \in \N} \subset \S$ denote a sequence of semimartingales with $X^{(k)}_0  = 0$. Then the following statements hold.
	\begin{enumerate}[label={\rm(\roman{*})}, ref={\rm(\roman{*})}] 
		\item\label{L:190729.i} If there exists $C>0$ such that $|\Delta X^{(k)}| \leq C$ 
		 for each $k\in \N$, and if 
		$
			\sum_{k  = 1}^\infty  [X^{(k)}, X^{(k)}]  < \infty,
		$
		then $\sum_{k  = 1}^\infty (X^{(k)} - B^{X^{(k)}})$ converges in the  $\S$--topology to a local martingale.
		\item\label{L:190729.ii} If $X^{(k)}$ has finite variation on compacts for each $k\in \N$ and if 
		$
			\sum_{k  = 1}^\infty \int_0^\cdot  |\d X^{(k)}|  < \infty,
		$
then $\sum_{k  = 1}^\infty X^{(k)}$ converges in the  $\S$--topology to a finite variation process.   
		\item\label{L:190729.iii}  Assume that $\sum_{k,l = 1}^\infty [ X^{(k)}, X^{(l)}]^- < \infty$. Then the following two statements are equivalent. 
			\begin{enumerate}[label={\rm(\Roman{*})}, ref={\rm(\Roman{*})}] 
			\item\label{L:190729.iii.I} $\sum_{k=1}^\infty [X^{(k)}, X^{(k)}]   < \infty$ and $\sum_{k  = 1}^\infty B^{X^{(k)}[1]}$ converges in the  $\S$--topology to a  process $B$. 
			\item\label{L:190729.iii.II} $\sum_{k=1}^\infty X^{(k)}$ converges in the $\S$--topology to a process $X$.
			\end{enumerate}
			If one (hence both) of these conditions hold then $B^{X[1]} = B$.  If additionally $[X^{(k)}, X^{(l)}] = 0$ for all $k,l \in \N$ with $k \neq l$ then  we also have $\sum_{k = 1}^\infty \Delta X^{(k)} = \Delta X$. 
	\end{enumerate}
\end{lemma}
\begin{proof}
	We first argue  \ref{L:190729.i}. By localization and by Lemma~\ref{L:190622.1}\ref{L:190622.1.i} we may  assume that there is a constant $\kappa \geq 0$ such that $\sum_{k=1}^\infty [X^{(k)}, X^{(k)}]\leq \kappa$. 
Next, fix for the moment $k \in \N$ and define the local martingale
$M^{(k)} = X^{(k)} - B^{ {X^{(k)}}}$.  Let $(\tau_m)_{m \in \N}$ be a nondecreasing sequence of stopping times such that $[M^{(k)}, B^{ {X^{(k)}}}]^{\tau_m}$ is a uniformly integrable martingale for each  $m \in \N$. Then we have
	\begin{align*}
		 \E\left[\left[M^{(k)}, M^{(k)}\right]_\infty\right] &{}= \lim_{m \uparrow \infty}
		 	\E\left[\left[M^{(k)}, M^{(k)}\right]_{\tau_m}\right]  \\
			&{}\leq  \lim_{m \uparrow \infty}
		 	\E\left[\left[X^{(k)}, X^{(k)}\right]_{\tau_m}\right]  = \E\left[\left[X^{(k)}, X^{(k)}\right]_{\infty}\right] .
	\end{align*}
	The Burkholder-Davis-Gundy inequality  
	now yields a  constant $\kappa' >0$ such that  
	\begin{align*}
		\sum_{k = 1}^\infty \E\left[\left(M^{(k)}_\infty\right)^2\right] &\leq \kappa' \sum_{k = 1}^\infty \E\left[\left[M^{(k)}, M^{(k)}\right]_\infty\right] 
		\leq  \kappa'\sum_{k = 1}^\infty \E\left[\left[X^{(k)}, X^{(k)}\right]_\infty\right]  \leq  \kappa' \kappa.
	\end{align*}
	Hence, Doob's inequality yields that $\sum_{k = 1}^\infty M^{(k)}$ converges locally in $\Hcal_{2}$ to a martingale; see, for example, 
	Kunita and Watanabe \cite[Proposition~4.1]{kunita.watanabe.67}
	or Dol\'eans-Dade and Meyer \cite[Lemme~1]{doleans-dade.meyer.70}. Hence by \cite[Theoreme~2]{emery.79}, \ref{L:190729.i} follows. 
	
	Let us now argue \ref{L:190729.ii}. First, $\sum_{k  = 1}^\infty X^{(k)}$  converges  to a finite variation process $X$ in the sense of uniform convergence on compacts in probability. Next,  note that 
	\[
		\zeta \cdot \sum_{k  = 1}^n X^{(k)} - \zeta \cdot X  \leq  \sum_{k = n+1}^\infty \int_0^\cdot \left|\d X^{(k)}\right|
	\]
	for all predictable processes $\zeta$ with $|\zeta| \leq 1$.  Hence,  \ref{L:190729.ii} follows.
	  
	 To see the implication from \ref{L:190729.iii.I} to \ref{L:190729.iii.II} in \ref{L:190729.iii}, apply \ref{L:190729.i} to the sequence $(X^{(k)}[1])_{k \in \N}$ and \ref{L:190729.ii} to $(x \indicator{|x| > 1} * \mu^{X^{(k)}})_{k \in \N}$.
	 For the reverse direction \ref{L:190729.iii.II} to \ref{L:190729.iii.I} note that since $X$ is a semimartingale, the assumption and  Lemma~\ref{L:190622.1}\ref{L:190622.1.iv} yield directly that $\sum_{k=1}^\infty [X^{(k)}, X^{(k)}]   < \infty$. Moreover, as above, the sums corresponding to $(X^{(k)}[1] -  B^{X^{(k)}[1]})_{k \in \N}$ and  $(x \indicator{|x| > 1} * \mu^{X^{(k)}})_{k \in \N}$ converge in the $\S$--topology; hence so must the sums corresponding to $ (B^{X^{(k)}[1]})_{k \in \N}$. Finally, if  \ref{L:190729.iii.I} and \ref{L:190729.iii.II} hold then 
	 \[
	 	X[1] =  \sum_{k  = 1}^\infty \left(X^{(k)}[1] - B^{X^{(k)}[1]}\right) +  B
	 \]
	 in the $\S$--topology,  where the first term is a local martingale by \ref{L:190729.i} and $B$ may be assumed to be predictable and of finite variation thanks to Lemma~\ref{L:190622.1}\ref{L:190622.1.v}\&\ref{L:190622.1.iv}. 
	 
	 Let us  additionally assume that  $[X^{(k)}, X^{(l)}] = 0$ for all $k,l \in \N$ with $k \neq l$.  
	 Then the sum $\sum_{k = 1}^\infty \Delta X^{(k)}$ is well defined since at most one summand is nonzero, $(\P \times \d A^X)$--a.e. By Lemma~\ref{L:190622.1}\ref{L:190622.1.iv} and the fact that  $\sum_{k=1}^\infty X^{(k)}$ converges to $X$ also in the sense of uniform convergence on compacts in probability  we may conclude.
	 \end{proof}
	
	\section{Replacing the \texorpdfstring{$\S$}{S}--topology by UCP convergence} \label{S:6}
Here we briefly discuss the choice of the $\S$--topology in the definition of $\J^2$.  Indeed, one may define an alternative class $\J^{2\dag} \subset \S$, where the convergence in the $\S$--topology is replaced by uniform convergence on compacts in probability (ucp). That is, a semimartingale $X \in \S$ is in $\J^{2\dag}$ if for a family $(\tau_k)_{k \in \N}$ of stopping times we have
\[
	\lim_{n \uparrow \infty} \E\left[\sup_{s \leq t} \left|X_s - X_0 - \sum_{k = 1}^n \Delta X_{\tau_k} \indicator{\lc \tau_k, \infty\lc}(s) \right| \wedge 1 \right] = 0
\]
for all $t \geq 0$.

Note that the equivalence of \ref{L:190622.2.I} and \ref{L:190622.2.II} in Lemma~\ref{L:190622.2} holds 
with convergence in $\S$--topology replaced by convergence in the sense of ucp in its statement. 
However, $\J^{2\dag}$ is not stable under $\sigma$--stopping, i.e., if $X \in \J^{2\dag}$ and $D$ is a predictable set then not necessarily $\indicator{D} \cdot X \in \J^{2\dag}$.

The new class  $\J^{2\dag}$ contains all elements of $\J^2$ because the semimartingale topology is stronger than ucp convergence. Proposition~\ref{P:190828} below shows $\J^{2\dag}$ is in fact too large for practical purposes or for representing `pure-jump' processes. 

\begin{proposition} \label{P:190828}
	Let $X$ denote a L\'evy process with $|\Delta X| \leq 1$, $\int x^+ F^X(\d x) =  \infty$, and symmetric and atomless L\'evy measure. Moreover, let $W$ denote an independent Brownian motion stopped when its absolute value hits $1$. Then $X + W \in \J^{2\dag}$; hence, in particular $\J^{2\dag} \setminus \J^1 \neq \emptyset$ for sufficiently rich probability spaces. 
\end{proposition}
\begin{proof}
	Fix for the moment $n \in \N$ and let $W^{(n)}$ denote a piecewise constant approximation of $W$ with $W^{(n)}_{k/n + t} = W_{k/n}$ for all $k \in \N$ and $t \in [0, 1/n)$.  Next,  let $B^{(n)} = \int_0^\cdot b^{(n)}_t \d t$ denote the trailing continuous piecewise linear predictable approximation of $W^{(n)}$.  By this we mean the  process $B^{(n)}$ such that $B^{(n)}_0 = B^{(n)}_{1/n} = 0$, $B^{(n)}_{2/n} = W^{(n)}_{1/n}$,  $B^{(n)}_{3/n} = W^{(n)}_{2/n}, \cdots$ and $b^{(n)}$ is  constant on each interval $[k/n, (k+1)/n)$ for $k \in \N$. Then it is clear that $\lim_{n \uparrow \infty} B^{(n)} = W$ in the sense of ucp.
	
	We now claim that there exist two nonincreasing sequences $(c^{(n)})_{n \in \N}$ and $(d^{(n)})_{n \in \N}$ of piecewise constant predictable processes with $c^{(n)}, d^{(n)} \in (0, 1/n]$ such that 
	$$\int |x| \indicator{\{x \notin (-g^{(n)}, f^{(n)})\}} F^X(\d x) < \infty
	\text{\quad and\quad}  
	x \indicator{\{x \notin (-g^{(n)}, f^{(n)})\}} * \nu^X = B^{(n)}.$$ 
	Then the statement follows by using the appropriate modification of Lemma~\ref{L:190622.2}.
	
To see the claim assume one has constructed $g^{(n)}$ and $f^{(n)}$ for some $n \in \N$ as required.  Consider now the intermediate predictable processes
$\overline{g} = g^{(n)} \wedge 1/(n+1)$ and $\overline{f} = f^{(n)} \wedge 1/(n+1)$ and the intermediate piecewise constant predictable process 
\[
	\overline b = b^{(n+1)} - \int x \indicator{\{x \notin (-\overline{g}, \overline{f})\}} F^X(\d x) .
\]
Whenever $\overline b > 0$, one now sets $g^{(n+1)} = \overline{g}$ and sets $f^{(n+1)}$ so that $\int x \indicator{\{x \in (f^{(n+1)}, \overline{f})\}} F^X(\d x) = \overline b$. When $\overline b < 0$ one sets  $g^{(n+1)}$  and $f^{(n+1)}$ in the opposite way. This construction satisfies the requirements, hence concluding the proof. 
\end{proof}

\section*{Acknowledgements}
The authors would like to thank an anonymous referee for critical comments that have lead to several improvements in the paper.

\end{document}